\documentclass[10pt]{amsart}
\usepackage{latexsym}
\usepackage{amsmath}
\usepackage{amssymb}
\usepackage{mathrsfs}
\usepackage{graphicx}
\usepackage{color}
\usepackage{pgfpages}
\usepackage{ifthen}
\usepackage{leftidx,tensor}
\usepackage[T1]{fontenc}
\usepackage[latin1]{inputenc}
\usepackage{mathtools}
\usepackage{comment}
\usepackage{dsfont}
\usepackage[nocompress]{cite}

\usepackage[shortlabels]{enumitem}
\usepackage{aliascnt}
\usepackage[bookmarks=true,pdfstartview=FitH, pdfborder={0 0 0}, colorlinks=true,citecolor=red, linkcolor=blue]{hyperref}
\usepackage{bbm}
\usepackage{nicefrac}

\theoremstyle{plain}
\newtheorem{thm}{Theorem}[section]
\newaliascnt{cor}{thm}
\newaliascnt{prop}{thm}
\newaliascnt{lem}{thm}
\newaliascnt{rem}{thm}
\newtheorem{cor}[cor]{Corollary}
\newtheorem{prop}[prop]{Proposition}
\newtheorem{lem}[lem]{Lemma}
\newtheorem{rem}[rem]{Remark}
\aliascntresetthe{cor}
\aliascntresetthe{prop}
\aliascntresetthe{lem}
\aliascntresetthe{rem} 
\theoremstyle{definition}
\newaliascnt{defn}{thm}
\newaliascnt{asu}{thm}
\newaliascnt{con}{thm}

\newtheorem{asu}[asu]{Assumption}

\aliascntresetthe{defn}
\aliascntresetthe{asu}
\aliascntresetthe{con}
\newcounter{stp}
\newcounter{stpi}
\newcounter{stpci}
\newcounter{stpiii}

\theoremstyle{remark}
\newaliascnt{exa}{thm}
\newaliascnt{masu}{thm}
\newaliascnt{nota}{thm}
\newaliascnt{sett}{thm}
\aliascntresetthe{exa}
\aliascntresetthe{masu}
\aliascntresetthe{nota}
\aliascntresetthe{sett}
\numberwithin{equation}{section}

\setlist[enumerate]{font = \normalfont}

\renewcommand {\S}	{\mathbb{S}}
\newcommand {\N}	{\mathbb{N}}

\newcommand {\R}	{\mathbb{R}}
\newcommand {\C}	{\mathbb{C}}
\newcommand {\E}	{\mathbb{E}}

\newcommand{\diag}[1]{\mathrm{diag}({#1})}
\renewcommand{\d}{\, \mathrm{d}}

\DeclareMathOperator{\Id}{Id}

\newcommand{\Hinfty}{\mathcal{H}^\infty}

\newcommand{\BIP}{\mathcal{BIP}}

\newcommand{\sA}{\mathcal{A}}
\newcommand{\tsA}{\tilde{\sA}}

\newcommand{\sE}{\mathcal{E}}
\newcommand{\sF}{\mathcal{F}}

\newcommand{\sX}{\mathcal{X}}
\newcommand{\sY}{\mathcal{Y}}
\newcommand{\sP}{\mathcal{P}}
\newcommand{\sL}{\mathcal{L}}

\newcommand{\sS}{\mathcal{S}}

\newcommand{\sR}{\mathcal{R}}

\renewcommand{\i}{\mathrm{i}}

\newcommand{\mre}{\mathrm{e}}

\newcommand{\D}{\mathrm{D}}
\renewcommand{\H}{\mathrm{H}}
\newcommand{\T}{\mathrm{T}}
\newcommand{\I}{\mathrm{I}}

\newcommand{\per}{\mathrm{per}}
\newcommand{\bc}{\mathrm{b.c.}}

\newcommand{\sigmabar}{\bar{\sigma}}
\newcommand{\sigd}{\sigma_\delta}

\newcommand{\atm}{\mathrm{atm}}
\newcommand{\ocn}{\mathrm{ocn}}
\newcommand{\ice}{\mathrm{ice}}

\newcommand{\ssbc}[1]{{#1}_{\bc}}

\newcommand{\sssigmabar}[1]{{#1}_{\sigmabar}}

\newcommand{\ssatm}[1]{{#1}_{\atm}}
\newcommand{\ssocn}[1]{{#1}_{\ocn}}
\newcommand{\ssice}[1]{{#1}_{\ice}}

\newcommand{\hatm}{\ssatm{h}}
\newcommand{\hocn}{\ssocn{h}}

\newcommand{\Omegaatm}{\ssatm{\Omega}}
\newcommand{\Omegaocn}{\ssocn{\Omega}}

\newcommand{\GaD}{\Gamma_{\D}}
\newcommand{\GaN}{\Gamma_{\mathrm{N}}}
\newcommand{\Gau}{\Gamma_u}
\newcommand{\Gab}{\Gamma_b}
\newcommand{\Gai}{\Gamma_i}
\newcommand{\Gao}{\Gamma_o}
\newcommand{\Gal}{\Gamma_l}
\newcommand{\Galatm}{\Gamma_{l,\atm}}
\newcommand{\Galocn}{\Gamma_{l,\ocn}}

\newcommand{\n}{\nu}
\newcommand{\ssn}[1]{\n_{#1}}

\newcommand{\nGao}{\ssn{\Gamma_o}}

\newcommand{\dk}[1]{\partial_{#1}}
\newcommand{\dt}{\dk{t}} 
\newcommand{\dz}{\dk{z}} 

\newcommand{\tr}{\mathrm{tr}}
\newcommand{\sstr}[1]{\tr_{#1}}
\newcommand{\trGau}{\sstr{\Gau}}
\newcommand{\trGab}{\sstr{\Gab}}
\newcommand{\trGai}{\sstr{\Gai}}
\newcommand{\trGao}{\sstr{\Gao}}

\newcommand{\ddn}{\dk{\n}}
\newcommand{\ssddn}[1]{\partial_{\ssn{#1}}}
\newcommand{\ddnGau}{\ssddn{\Gau}}

\newcommand{\ddnGai}{\ssddn{\Gai}}
\newcommand{\ddnGao}{\ssddn{\Gao}}

\newcommand{\ddnr}[1]{\ddn^{#1}}
\newcommand{\ddnq}{\ddnr{q}}
\newcommand{\ddnp}{\ddnr{q'}}

\newcommand{\ReHatm}{\mathrm{Re}_{\H,\atm}}
\newcommand{\Rezatm}{\mathrm{Re}_{z,\atm}}
\newcommand{\ReHocn}{\mathrm{Re}_{\H,\ocn}}
\newcommand{\Rezocn}{\mathrm{Re}_{z,\ocn}}

\newcommand{\Catm}{\ssatm{C}}

\newcommand{\muocn}{\ssocn{\mu}}

\newcommand{\Ratm}{\ssatm{R}}

\newcommand{\ratm}{\ssatm{\rho}}

\newcommand{\rice}{\ssice{\rho}}

\newcommand{\Coi}{C_{\mathrm{o,i}}}
\newcommand{\sAoi}{\sA_{\mathrm{o,i}}}

\newcommand{\eps}{\varepsilon}
\renewcommand{\phi}{\varphi}
\newcommand{\vatm}{\ssatm{v}}
\newcommand{\watm}{\ssatm{w}}
\newcommand{\vocn}{\ssocn{v}}
\newcommand{\wocn}{\ssocn{w}}
\newcommand{\patm}{\ssatm{\pi}}
\newcommand{\pocn}{\ssocn{\pi}}
\newcommand{\vice}{\ssice{v}}
\newcommand{\uatm}{\ssatm{u}}
\newcommand{\uocn}{\ssocn{u}}

\newcommand{\fatm}{\ssatm{f}}
\newcommand{\focn}{\ssocn{f}}
\newcommand{\fice}{\ssice{f}}

\newcommand{\tatm}{\ssatm{\tau}}
\newcommand{\tocn}{\ssocn{\tau}}

\newcommand{\Sa}{\mathrm{S}_\mathrm{a}}
\newcommand{\Sh}{\mathrm{S}_\mathrm{h}}

\newcommand{\xH}{x_\mathrm{H}}

\renewcommand{\bar}[1]{\overline{#1}}
\newcommand{\vbar}{\bar{v}}

\newcommand{\vtilde}{\tilde{v}}

\newcommand{\intatm}[1]{\int_{\Omegaatm} {#1} \d(\xH,z)}
\newcommand{\intocn}[1]{\int_{\Omegaocn} {#1} \d(\xH,z)}
\newcommand{\intO}[1]{\int_{\Omega} {#1} \d(\xH,z)}
\newcommand{\intpO}[1]{\int_{\partial\Omega} {#1} \d \rS}
\newcommand{\intG}[1]{\int_G {#1} \d \xH}

\newcommand{\scalarprod}[3]{\langle {#1}, {#2} \rangle_{#3}}
\newcommand{\scalarprodLtwo}[3]{\scalarprod{#1}{#2}{\rL^2({#3})}}
\newcommand{\scalarprodLtwoatm}[2]{\scalarprod{#1}{#2}{\rL^2(\Omegaatm)}}
\newcommand{\scalarprodLtwoocn}[2]{\scalarprod{#1}{#2}{\rL^2(\Omegaocn)}}
\newcommand{\scalarprodLtwoG}[2]{\scalarprod{#1}{#2}{\rL^2(G)}}

\renewcommand{\div}{\mathrm{div} \, }
\newcommand{\divH}{\mathrm{div}_{\H} \,}
\newcommand{\nablaH}{\nabla_{\H}}
\newcommand{\DeltaH}{\Delta_{\H}}

\newcommand{\Deltaa}{d_{\mathrm{a}}\DeltaH}
\newcommand{\Deltah}{d_{\mathrm{h}}\DeltaH}

\newcommand{\drh}{d_{\mathrm{h}}}
\newcommand{\dra}{d_{\mathrm{a}}}
\newcommand{\frh}{f_{\mathrm{h}}}
\newcommand{\fra}{f_{\mathrm{a}}}
\newcommand{\crh}{c_{\mathrm{h}}}
\newcommand{\cra}{c_{\mathrm{a}}}
\renewcommand{\th}{\Tilde{h}}
\newcommand{\ta}{\Tilde{a}}
\newcommand{\oh}{\overline{h}}
\newcommand{\oa}{\overline{a}}

\newcommand{\tri}{\triangle}
\newcommand{\trid}{\triangle_\delta}

\newcommand{\Sd}{S_\delta}
\newcommand{\hra}{\hookrightarrow}

\newcommand{\rC}{\mathrm{C}}
\newcommand{\rL}{\mathrm{L}}
\newcommand{\rW}{\mathrm{W}}
\newcommand{\rH}{\H}
\newcommand{\rB}{\mathrm{B}}

\newcommand{\rF}{\mathrm{F}}

\newcommand{\rS}{\mathrm{S}}

\newcommand{\rLsigmabar}{\sssigmabar{\rL}}

\newcommand{\rLq}{\rL^q}
\newcommand{\rLp}{\rL^p}
\newcommand{\rHsq}[1]{\rH^{{#1},q}}
\newcommand{\rWsq}[1]{\rW^{{#1},q}}
\newcommand{\rWsp}[1]{\rW^{{#1},p}}
\newcommand{\rBrpq}[3]{\rB^{{#1}}_{{#2}{#3}}}
\newcommand{\rBsqp}[1]{\rBrpq{#1}{q}{p}}

\newcommand{\rHsqper}[1]{\rH_\per^{{#1},q}}

\newcommand{\rBsqpper}[1]{\rB^{{#1}}_{{q}{p},\per}}

\newcommand{\rLz}[1]{\rL_z^{#1}}
\newcommand{\rLxy}[1]{\rL_{xy}^{#1}}
\newcommand{\rWz}[2]{\rW_z^{{#1},{#2}}}
\newcommand{\rWxy}[2]{\rW_{xy}^{{#1},{#2}}}

\newcommand{\rBrpqz}[3]{\rB^{{#1}}_{{#2}{#3},{z}}}
\newcommand{\rBrpqxy}[3]{\rB^{{#1}}_{{#2}{#3},{xy}}}

\newcommand{\rHsqbcper}[1]{\rH^{{#1},q}_{\per,\bc}}

\newcommand{\rLqsigmabar}{\rLsigmabar^q}

\newcommand{\rLqOmegasigmabar}{\rLqsigmabar(\Omega)}
\newcommand{\rLqatm}{\rLqsigmabar(\Omegaatm)}
\newcommand{\rLrocn}[1]{\sssigmabar{\rL}^{#1}(\Omegaocn)}
\newcommand{\rLqocn}{\rLrocn{q}}
\newcommand{\rLtwoocn}{\rLrocn{2}}
\newcommand{\rLpocn}{\rLrocn{q'}}

\newcommand{\rLratm}[1]{\sssigmabar{\rL}^{#1}(\Omegaatm)}
\newcommand{\rLtwoatm}{\rLratm{2}}
\newcommand{\rLtwoG}{\rL^2(G)}
\newcommand{\rLtwoOatm}{\rL^2(\Omegaatm)}

\newcommand{\rLinftyG}{\rL^\infty(G)}

\renewcommand{\P}{\mathcal{P}}
\newcommand{\PH}{\P_{\H}}
\newcommand{\Patm}{\ssatm{\P}}
\newcommand{\Pocn}{\ssocn{\P}}

\newcommand{\A}{\mathrm{A}}
\newcommand{\Aatm}{\mathrm{A}^{\atm}}
\newcommand{\Aocn}{\mathrm{A}^{\ocn}}
\newcommand{\An}{\Aocn_0}
\newcommand{\Anr}[1]{\Aocn_{0,#1}}
\newcommand{\Anq}{\Anr{q}}
\newcommand{\Anp}{\Anr{q'}}

\newcommand{\Am}{\mathrm{A}_{\mathrm{m}}}
\newcommand{\Abc}{\ssbc{\mathrm{A}}}

\newcommand{\Amocn}{\Aocn_{\mathrm{m}}}

\newcommand{\AH}{\mathrm{A}^{\H}}

\newcommand{\sAH}{\mathcal{A}^{\H}}

\newcommand{\Llambda}[1]{\mathrm{L}_{#1}}
\newcommand{\Nlambda}[1]{\mathrm{N}_{#1}}

\newcommand{\Aice}{\mathrm{A}^\ice}
\newcommand{\Aiceom}{\mathrm{A}_\omega^\ice}

\newcommand{\B}{\mathrm{B}}
\newcommand{\Ba}{\B_a}
\newcommand{\Bh}{\B_h}

\newcommand{\J}{\mathrm{J}}
\newcommand{\Je}{\J_1}
\newcommand{\Jz}{\J_2}

\newcommand{\bilinu}[2]{({#1}\cdot \nabla) {#2}}

\newcommand{\bilinH}[2]{( {#1} \cdot \nablaH ) {#2}}
\newcommand{\bilinz}[2]{w({#1}) \cdot \dz  {#2}}
\newcommand{\bilin}[2]{\bilinH{#1}{#2}+\bilinz{#1}{#2}}

\newcommand{\bilinHice}{\bilinH{\vice}{\vice}}

\newcommand{\bilinocn}{\bilin{\vocn}{\vocn}}
\newcommand{\bilinatm}{\bilin{\vatm}{\vatm}}

\newcommand{\Fice}{\mathrm{F}^\ice}
\newcommand{\Ficeom}{\mathrm{F}_\omega^\ice}

\newcommand{\Vatm}{V^\atm}
\newcommand{\Vocn}{V^\ocn}
\newcommand{\Vice}{V^\ice}

\newcommand{\Watm}{W^\atm}
\newcommand{\Wocn}{W^\ocn}
\newcommand{\Wice}{W^\ice}

\newcommand{\sYzero}{\sY_0}
\newcommand{\sYzeroa}{\sY_0^\atm}
\newcommand{\sYzeroo}{\sY_0^\ocn}
\newcommand{\sYzeroi}{\sY_0^\ice}
\newcommand{\sYone}{\sY_1}

\newcommand{\sYbeta}{\sY_\beta}

\newcommand{\sYgamma}{\sY_\gamma}
\newcommand{\sYgammaa}{\sY_\gamma^\atm}
\newcommand{\sYgammao}{\sY_\gamma^\ocn}
\newcommand{\sYgammai}{\sY_\gamma^\ice}
\newcommand{\sYtheta}{\sY_\theta}
\newcommand{\sYthetaa}{\sY_\theta^\atm}
\newcommand{\sYthetao}{\sY_\theta^\ocn}
\newcommand{\sYthetai}{\sY_\theta^\ice}

\newcommand{\oB}{\overline{B}}

\newcommand{\kone}{\kappa_1}
\newcommand{\ktwo}{\kappa_2}

\newcommand{\Fatm}{\ssatm{\mathrm{F}}}
\newcommand{\Focn}{\ssocn{\mathrm{F}}}

\newcommand{\on}[1]{\text{ on } {#1}}
\renewcommand{\for}{\text{ for }}

\newcommand{\andtext}{\text{ and }}

\newcommand{\qandq}{\quad \text{and} \quad}

\newcommand{\onGT}{\on{G \times [0,T]}}

\newcommand{\onOmega}{\on{\Omega}}

\newcommand{\onOmegaT}{\on{\Omega \times [0,T]}}

\newcommand{\onOmegaatmT}{\on{\Omegaatm \times [0,T]}}

\newcommand{\onOmegaocn}{\on{\Omegaocn}}
\newcommand{\onOmegaocnT}{\on{\Omegaocn \times [0,T]}}

\title{Interaction of Geophysical Flows  with Sea Ice Dynamics}
\author{Tim Binz}
\address{Princeton University,
Program in Applied \& Computational Mathematics, Fine Hall, Washington Road, 08544 Princeton, NJ, USA.}
\email{tb7523@princeton.edu}
\author{Felix Brandt}
\address{University of California, Berkeley, Department of Mathematics, Evans Hall, 94720 Berkeley, CA, USA.}
\email{fbrandt@berkeley.edu}
\author{Matthias Hieber}
\address{Technische Universit\"{a}t Darmstadt,
Schlo\ss{}gartenstra{\ss}e 7, 64289 Darmstadt, Germany.}
\email{hieber@mathematik.tu-darmstadt.de}


\dedicatory{Dedicated to the memory of Guiseppe Da Prato}

\begin{document}
\begin{abstract}
This article establishes local strong well-posedness and global strong well-posedness close to constant equilibria of a model coupling the primitive equations of ocean and atmosphere dynamics with Hibler's viscous-plastic sea ice model. 
In order to treat the coupling conditions, an approach involving the  hydrostatic Dirichlet and Dirichlet-to-Neumann operator is developed. 
Mapping properties of the  latter operators are investigated for the first time and are of central importance for showing that the operator associated with the linearized coupled system admits a bounded $\Hinfty$-calculus on suitable $\rL^q$-spaces. 
Quasilinear methods allow then to obtain the strong well-posedeness results described above. 
\end{abstract}

\maketitle

\section{Introduction}
\label{sec:intro}

In a series of articles, Lions, Teman and Wang developed a coupled atmosphere-ocean model, the so-called CAO-model, describing the atmosphere and the ocean by the primitive equations coupled by nonlinear boundary conditions, see  \cite{LTW:92a, LTW:92b, LTW:93, LTW:95}. 
Their work can be regarded as the starting point of the mathematical analysis of a large class of geophysical flow models.
In fact, Lions, Temam and Wang analyzed the CAO-model consisting of two primitive equations coupled by wind driven boundary conditions and proved the existence of weak solutions to this model.

Starting from there, many authors studied the equations describing in particular the dynamics of the ocean by the incompressible primitive equations. 
In a seminal paper, Cao and Titi \cite{CT:07} established global strong well-posedness of the incompressible primitive equations for initial data in $\rH^1$ by energy methods.  
For related results, we refer to the work of Kukavica and Ziane \cite{KZ:07}.
A different  approach to the primitive equations by means of evolution equations was developed by Hieber and Kashiwabara \cite{HK:16} as well as Giga, Gries, Hieber, Hussein and Kashiwabara \cite{GGHHK:20, GGHHK:20b, GGHHK:21}.

The primitive equations subject to {\em stochastic wind driven boundary conditions} have been investigated in \cite{BHHS:22}, adapting an approach by Da Prato and Zabcyk \cite{DPZ:96} for stochastic boundary
conditions to the situation of geophysical flows.
More precisely, consider the primitive equations, where the  standard deterministic boundary conditions  are replaced by a stochastic
boundary condition modeling the wind as $\dz V = h_b \dt \omega$, with a function $h_b$ defined at the boundary and $\omega(t) = \sum_{n=1}^\infty \langle g,e_n \rangle W(t)e_n$ for a suitable function $g$ and a cylindrical Wiener process $W$ on a separable Hilbert space $H$ with orthonormal basis $(e_n)$.
Following Da Prato and Zabcyk \cite{DPZ:96}, one can express a solution $V$ to the primitive equations subject to the above stochastic boundary condition by a solution to the equation
\begin{equation*}
    \d Z(t) + A Z(t) \d t = A[Nh_b(t)g] \d W(t), \quad Z(0)=Z_0,
\end{equation*}
subject to deterministic boundary conditions. 
Here $N$ is the so-called Neumann operator mapping deterministic inhomogeneous boundary data to the solution of the associated stationary hydrostatic Stokes problem, and $A$ denotes the hydrostatic Stokes operator. 
A similar method has been used in \cite{HHS:23}.   

In this article, we turn our attention to an extension of the above CAO-model. 
Indeed,  we include sea ice dynamics into the CAO-model and consider a coupled atmosphere-sea ice-ocean model.
We then analyze this new coupled model, called the CIAO-model, and establish local strong well-posedness for large data and global strong well-posedness for data close to constant equilibria.  

Large scale sea ice dynamics is very often described by using Hibler's model \cite{Hib:79}. 
There the dynamics of the velocity, the thickness and the concentration of sea ice is described by a set of quasilinear equations.
They were investigated numerically by different communities, see e.\ g.\ \cite{MK:21,SK:18}. For a survey on the modeling of sea ice, we refer to \cite{Gol:20}. 
Let us note that the rigorous analysis of the Hibler's sea ice model began only recently by the article of Brandt, Disser, Haller-Dintelmann and Hieber \cite{BDHH:22}, see also \cite{Bra:25}, as well as the article of Liu, Thomas and Titi \cite{LTT:22}. 
The associated system of equations is coupled, degenerate, quasilinear and parabolic-hyperbolic. 
The interaction problem of sea ice with a rigid body has been studied in \cite{BBH:22}.

In the present CIAO-model, the primitive equations for the ocean and the atmosphere are coupled with Hibler's quasilinear equations for sea ice. 
We investigate the interaction problem of the atmosphere and the ocean when additionally considering sea ice as a thin layer coupled to the atmosphere and the ocean. 
For a precise formulation of the coupling conditions, see  \eqref{eq:tatm}, \eqref{eq:tocn} and \eqref{eq:coupling} in \autoref{sec:coupled model and main result}. 

Our approach to the CIAO-model can be described as follows: 
Motivated by the quasilinear structure of Hibler's model, we write the CIAO-model as a quasilinear evolution equation, subject to coupling conditions. 
In order to treat these coupling conditions, inspired by the study of the Neumann operator $N$ as in \cite{BHHS:22}, we develop a decoupling approach involving the  hydrostatic Dirichlet and Dirichlet-to-Neumann operator. 
Such operators have been studied intensively in the parabolic situation and also for fluids but not for geophysical flows as the primitive equations.
Our well-posedness results of the stationary hydrostatic Stokes problem with inhomogeneous boundary conditions shown in \autoref{sec:stationary problem} are of central importance. 
Indeed, the hydrostatic Dirichlet and Dirichlet-to-Neumann operators allow to show that the linearized operator associated with the coupled system admits a bounded $\Hinfty$-calculus on suitable $\rL^q$-spaces. 
This enables us to use modern quasilinear methods for solving the  coupled system in the strong setting.

At this point, some comments on the background of these quasilinear methods are in order.
The so-called {\em Da Prato-Grisvard theorem} \cite{DPG:75} on the sum of two operators is a classical milestone in the theory of maximal regularity of solutions of evolution equations. 
It reads as follows:
Let~$A$ be the generator of a bounded analytic semigroup $(\mre^{tA})_{t\geq 0}$ on a Banach space $X$ with domain $D(A)$ such that $0 \in \rho(A)$, and consider $\theta \in (0,1)$ as well as  $1 \le q < \infty$. 
Then there exists a constant  $C>0$ such that for all $f \in \rL^q(\R_+;D_A(\theta,q))$, the function $u$ given for $t>0$ by $u(t):=\int_0^t \mre^{(t-s)A}f(s)\d s$ satisfies the estimate
\begin{equation*}
    \|Au \|_{\rL^q(\R_+;D_A(\theta,q))} \leq C \cdot \|f\|_{\rL^q(\R_+;D_A(\theta,q))}. 
\end{equation*}
Here $D_A(\theta,q)$ is defined by $D_A(\theta,q)= \{x \in X: [x]_{\theta,q}:= (\int_0^\infty \|t^{1-\theta} A\mre^{tA}x\|_X^q \nicefrac{\d t}{t})^{1/q} < \infty \}$.
When equipped with the norm $\|x\|_{\theta,q}=\|x\|_X + [x]_{\theta,q}$, the space $D_A(\theta,q)$ becomes a Banach space and coincides with the real interpolation space $(X,D(A))_{\theta,q}$. 
In addition, if $A$ is invertible, then the norm of the real interpolation space is equivalent to the homogeneous norm $[\cdot]_{\theta,q}$. 
This result acted then as a starting point for the very many developments concerning the maximal regularity  approach to evolution equations.
We refer here e.\ g.\ to the work of Da Prato and Grisvard \cite{DPG:79} on abstract nonlinear evolution equations.

Motivated free boundary problems in fluid dynamics, a homogeneous version of the {Da~Prato-Grisvard~theorem} was developed in \cite{DHMT:22}. 
More precisely, if $D_A(\theta,q)$ is endowed with the homogeneous norm $[\cdot]_{\theta,q}$, then a global maximal $L^q$-regularity result was deduced for $1\leq q <\infty$ {\em without assuming invertibility} of the generator of the associated semigroup.
This yields in particular a framework allowing for global-in-time control of the change of Eulerian to Lagrangian coordinates, which is decisive  in various free boundary problems in  fluid dynamics, see \cite{DHMT:22} for more information. 
In particular, the case $q=1$ is of crucial importance for global existence results for certain free boundary problems in the critical space
$\rL^1(\R_+;\dot{\rB}^s_{p,1}(\R_+))$.

Further important applications of the classical Da Prato-Grisvard theorem \cite{DPG:75} concern time periodic solutions to quasilinear evolution equations.
In \cite{BH:23}, the Da Prato-Grisvard theorem was extended to the quasilinear setting to show the existence of a unique time periodic strong solution to quasilinear evolution equations in $D_{A(0)}(\theta,q)$ provided the underlying linearized operator $A(0)$ generates a bounded analytic semigroup on $X$ with $0 \in \rho(A(0))$, the nonlinearities satisfy certain regularity and Lipschitz conditions, and the periodic forcing terms are sufficiently small.
For further applications, see also \cite{HKKT:20}. 

The structure of this article is as follows:
In \autoref{sec:prelim}, we introduce the  notation and give basic concepts used throughout the paper.
\autoref{sec:coupled model and main result} is dedicated to deriving the coupling conditions as well as the complete coupled atmosphere-sea ice-ocean system, and to stating the two main results 
on the local strong well-posedness and the global strong well-posedness close to constant equilibria, respectively.
We then reformulate the coupled system of equations \eqref{eq:coupled system} as a quasilinear evolution equation in \autoref{sec:coupled system as qee}.
The sea ice equations on the quadratic domain $(0,1) \times (0,1)$ are investigated in \autoref{sec:Hibler}, while we recall properties of the primitive equations on cylindrical domains in \autoref{sec:pe}.
The central \autoref{sec:stationary problem} deals with the stationary hydrostatic Stokes problem and the hydrostatic Dirichlet-to-Neumann operator. 
In \autoref{sec:decoupling}, we present the argument to prove the bounded $\Hinfty$-calculus of the coupled operator, and we show Lipschitz estimates of the coupled system.
\autoref{sec:proofs local} and \autoref{sec:proofs global} are  concerned with the proof of the  main results.

\section{Preliminaries}\label{sec:prelim}

We use $x$ and $y$ as well as $\xH := (x,y)$ to denote horizontal coordinates, whereas $z$ is employed for the vertical coordinate and $t$ is the time variable.
The sub- or superscripts $o^\atm$, $o^\ocn$ and $o^\ice$ are used to indicate whether objects $o$ in the context of the atmosphere, the ocean or sea ice are considered.
For instance, $\uatm$, $\vatm$ and $\watm$ represent the full, horizontal and vertical velocity of the atmospheric wind, $\uocn$, $\vocn$ and $\wocn$ denote the full, horizontal and vertical velocity of the ocean, and $\vice$ is the horizontal velocity of the sea ice.
With $v = (\vice,h,a)$, the principle variable of the system will be denoted by $u = (\vatm,\vocn,v) = (\vatm,\vocn,\vice,h,a)$, while the respective pressure terms for the atmosphere and the ocean will be denoted by $\patm$ and $\pocn$.
For the sea ice part, with $G:=(0,1) \times (0,1)$, we also introduce the mean ice thickness $h \colon G \times [0,T] \to (\kone,\ktwo)$ and the mean ice compactness $a \colon G \times [0,T] \to (0,1)$, defined as the ratio of thick ice per area.
We remark that $\kone > 0$ sufficiently small is a parameter indicating the transition to open water, i.\ e., a value of $h(x,y,t)$ less than $\kone$ means that at $(x,y) \in G$ at time $t$ there is open water.
In contrast, $\ktwo > 0$ sufficiently large denotes an upper bound for the mean ice thickness.

Unless stated otherwise, domains will be denoted by $\Omega$, possibly with a sub- or superscript, while $\Gamma$ represents boundaries.
The domains under consideration are $\Omegaatm = G \times (\ktwo,\hatm)$ for the atmosphere, where $\ktwo > 0$ is the upper bound for the mean ice thickness, as well as $\Omegaocn = G \times (-\hocn,0)$ for the ocean.
The upper boundary is denoted by $\Gau := G \times \{\hatm\}$, the lower boundary by $\Gab := G \times \{-\hocn\}$, and the interfaces between the ocean and the sea ice and the sea ice and the atmosphere are denoted by $\Gao := G \times \{0\}$ and $\Gai := G \times \{\ktwo\}$, respectively.
We remark that these boundaries and interfaces will be identified with $G$ in the sequel.
In addition, $\Galatm := \partial G \times (\ktwo,\hatm)$ represents the lateral boundary associated to the atmosphere, whereas $\Galocn := \partial G \times (-\hocn,0)$ denotes the lateral boundary of the ocean.
The thickness of the sea ice, expressed by $h$, is relatively small compared to the heights of the ocean and the atmosphere.
For simplicity, we assume that the interfaces between the sea ice and the ocean as well as the atmosphere are flat in this article.
The more involved consideration of free surfaces is left to future studies.

Furthermore, $\Delta = \dk{x}^2 + \dk{y}^2 + \dk{z}^2$, $\nabla = (\dk{x}, \dk{y}, \dk{z})^\top$ and $\div f = \dk{x} f_1 + \dk{y} f_2 + \dk{z} f_3$ represent the Laplacian, the gradient and the divergence of a vector field $f \colon \R^3 \to \R^3$, while the respective horizontal objects are denoted by $\DeltaH = \dk{x}^2 + \dk{y}^2$, $\nablaH = (\dk{x}, \dk{y})^\top$ and $\divH g = \dk{x} g_1 + \dk{y} g_2$ for a vector field $g \colon \R^2 \to \R^2$.
The trace and the normal derivative on some boundary $\Gamma$ will be denoted by $\tr_\Gamma$ and $\ssddn{\Gamma}$, respectively.
As it is only used in the context of the sea ice, the deformation tensor $\eps$ associated to sea ice is denoted by $\eps = \eps(\vice) = \frac{1}{2}\left(\nablaH \vice + (\nablaH \vice)^\top\right)$.

Consider a domain $\Omega \subset \R^n$, $p,q \in [1,\infty]$, $k \in \N$ and $s \ge 0$.
We denote by $\rLq(\Omega)$ the classical $\rLq$-spaces, by $\rWsq{k}(\Omega)$ or $\rHsq{k}(\Omega)$ the classical Sobolev spaces of order $k$, by $\rHsq{s}(\Omega)$ the Bessel potential spaces, by $\rWsq{s}(\Omega)$ the fractional Sobolev spaces or Sobolev-Slobodeckij spaces and by $\rBrpq{s}{q}{p}(\Omega)$ the Besov spaces.
For details on these spaces, we refer e.\ g.\ to the monographs \cite{Tri:78} or \cite{Ama:19}.
If $\Omega$ is a bounded domain with sufficiently smooth boundary $\Gamma$, we use the subscript $\bc$ to indicate that the respective space is considered with boundary conditions.

For $\Omega = G \times (a,b)$, $-\infty < a < b < \infty$, we denote by $\vbar$ the vertical average of $v$, i.\ e., $\vbar := \frac{1}{b-a} \int_a^b v(\cdot,\cdot,\xi) \d \xi$.
The primitive equations are modeled on the hydrostatic solenoidal $\rLq$-spaces given by the $\rLq$-closure of the smooth hydrostatic functions, so following the approach developed in \cite[Sections~3 and 4]{HK:16}, we set
\begin{equation*}
	\rLqsigmabar(\Omega) 
	:= \bar{\{ v \in \rC_\per^\infty(\overline{\Omega})^2 \colon \divH\vbar = 0 \}}^{\| \cdot \|_{\rLq(\Omega)}},
\end{equation*} 
and horizontal periodicity is incorporated by the function spaces $\rC_\per^\infty(\overline{\Omega})$ and $\rC_\per^\infty(\overline{G})$.
In that respect, for $p,q \in (1,\infty)$ and $s \in [0,\infty)$, we define the Bessel potential spaces with horizontal periodicity by
\begin{equation*}
    \rHsqper{s}(\Omega) := \bar{\rC_\per^\infty(\overline{\Omega})}^{\| \cdot \|_{\rHsq{s}(\Omega)}} \qandq \rHsqper{s}(G) := \bar{\rC_\per^\infty(\overline{G})}^{\| \cdot \|_{\rHsq{s}(G)}},
\end{equation*}
and the Besov spaces with horizontal periodicity $\rBsqpper{s}(\Omega)$ and $\rBsqpper{s}(G)$ are defined analogously.
The Bessel potential spaces $\rHsq{s}(\Omega)$ and Besov spaces $\rBsqp{s}(\Omega)$ are defined as restrictions of the respective spaces on the whole space to $\Omega$. 
We remark that $\rHsqper{0} := \rLq$.

\section{Main Results for the Coupled System}
\label{sec:coupled model and main result}

The main assumptions with regard to the coupling of the atmosphere and the ocean with sea ice are that the first two exert the forces $\tatm$ and $\tocn$ on the sea ice, and that the velocity of the sea ice coincides with the horizontal velocity of the ocean at the water surface.
More precisely, we assume that the force exerted by the atmosphere is given by
\begin{equation}
	\tatm = \ratm \Catm |\trGai \vatm | \Ratm \trGai \vatm, \on{G}.
	\label{eq:tatm}
\end{equation}

The force exerted by the ocean on the sea ice is supposed to be proportional to the shear rate, so it is of the form 
\begin{equation}
	\tocn = -\muocn \ddnGao \vocn, \on{G}.
	\label{eq:tocn}
\end{equation}
This is in accordance with a plane Couette flow for a Newtonian fluid, where the stress tensor is given by $\T - \pi \I_3$, with $\T$ denoting the usual Cauchy stress tensor, 
compare e.\ g.\ \cite[Section~7.2]{Rud:19}.
The condition on the equality of the velocities on the interface of the ocean and the sea ice can be expressed by
\begin{equation}
	\trGao \vocn = \vice, \on{G}.
	\label{eq:coupling} 
\end{equation}
In the above, $\ratm$ represents the density for air, $\Catm$ denotes an air drag coefficient, and $\Ratm$ is a rotation matrix operating wind vectors, while $\muocn > 0$ is the viscosity associated to the fluid described by the primitive equations of the ocean.

Concerning the coupling conditions, the form of the force due to the atmospheric wind in \eqref{eq:tatm} is taken from Hibler's classical article \cite{Hib:79}.
For the coupling of the sea ice and the ocean, it is assumed, as written above, that the force exerted by the ocean on the sea ice is proportional to the shear rate in \eqref{eq:tocn}, which
corresponds to a balance of forces at the sea ice-ocean interface.  The continuity of the sea ice and ocean velocities at the interface as expressed in \eqref{eq:coupling} amounts to saying that sea ice
and ocean move with the same velocity at the interface. This is inspired by the consideration of fluid-structure interaction problems where the structure is located at a part of the
fluid boundary.

Before providing the complete coupled system, we briefly discuss the internal ice stress.
We follow \cite{Hib:79} and refer e.\ g.\ to \cite[Section~2]{BDHH:22} for a more thorough treatment.
In fact, the viscous-plastic rheology is expressed by a constitutive law linking the internal ice stress $\sigma$ and the deformation tensor $\eps$ via an internal ice strength $P$ as well as nonlinear bulk and shear viscosities such that the principal components of the stress lie on an elliptical yield curve.
With $e > 1$ denoting the ratio of major to minor axes, the constitutive law is given by
\begin{equation*}
    \sigma = \frac{1}{e^2} \frac{P}{\tri(\eps)} \eps + \left(1 - \frac{1}{e^2}\right) \frac{P}{2 \tri(\eps)} \tr(\eps) \I_2 - \frac{P}{2}\I_2,
\end{equation*}
where for the ice thickness $h$, the ice compactness $a$ and given constants $p^* > 0$ and $c > 0$, the ice strength $P$ takes the shape $P = P(h,a) = p^* h \exp(-c(1-a))$.
Moreover, we have
\begin{equation*}
    \tri^2(\eps) := \left(\eps_{11}^2 + \eps_{22}^2\right) \left(1 + \frac{1}{e^2}\right) + \frac{4}{e^2} \eps_{12}^2 + 2 \eps_{11} \eps_{22} \left(1 - \frac{1}{e^2}\right).
\end{equation*}

Even though the above law represents an idealized viscous-plastic material, the viscosities become singular for $\tri$ tending to $0$.
Following \cite{BDHH:22, MK:21}, for $\delta > 0$, we take the regularization $\trid(\eps) := \sqrt{\delta + \tri^2(\eps)}$ into account and define the regularized ice stress by
\begin{equation*}
    \sigd := \frac{1}{e^2} \frac{P}{\trid(\eps)} \eps + \left(1 - \frac{1}{e^2}\right) \frac{P}{2 \trid(\eps)} \tr(\eps) \I_2 - \frac{P}{2}\I_2.
\end{equation*}

We follow \cite{LTW:92a, LTW:92b, LTW:95} for the incompressible, viscous primitive equations of the atmosphere and the ocean as well as \cite{Hib:79} for the sea ice equations.
The complete coupled system is then given by

\begin{equation}
\left\{
	\begin{aligned}
		\dt \vatm - \Delta \vatm 
		+ \bilinu{\uatm}{\vatm}
		+ \nablaH \patm &= \fatm, &&\onOmegaatmT, \\
		\dz \patm &= 0, &&\onOmegaatmT, \\
		\div \uatm &= 0, &&\onOmegaatmT, \\
		\dt \vocn - \Delta \vocn 
		+ \bilinu{\uocn}{\vocn}
		+ \nablaH \pocn &= \focn, &&\onOmegaocnT, \\
		\dz \pocn &= 0, &&\onOmegaocnT, \\
		\div \uocn &= 0, &&\onOmegaocnT, \\
		\dt\vice
		- \frac{1}{m} \divH \sigd + \bilinHice
		&= \frac{1}{m} \tatm + \frac{1}{m} \tocn - g \nablaH H, &&\onGT, \\
		\dt h -\Deltah h + \divH(\vice h) &= \Sh, &&\onGT, \\
		\dt a - \Deltaa a + \divH(\vice a) &= \Sa, &&\onGT, \\
		\trGao \vocn &= \vice, &&\onGT.
	\end{aligned}
\right. 	
	\label{eq:coupled system}
\end{equation}

In the above, $\fatm$ as well as $\focn$ represent external forces, $m = \rice h$ denotes the mass of the sea ice, where $\rice > 0$ is the density, $g$ denotes the gravity and $H \colon G \times [0,T] \to [0,\infty)$ the sea surface dynamic height, $\Sh$ and $\Sa$ are thermodynamic terms discussed in more detail in \autoref{sec:Hibler}, and $\drh$, $\dra > 0$ are constants.
For simplicity, we omit Reynolds numbers and terms associated to the Coriolis force in \eqref{eq:coupled system}, but we emphasize that they can be incorporated easily, see \autoref{cor:possible extenions main result}.

The system is completed by Neumann boundary conditions for the horizontal velocity of the atmospheric wind $\vatm$, Dirichlet boundary conditions for the horizontal velocity of the ocean $\vocn$ on the lower boundary $\Gab$ and Dirichlet boundary conditions for the vertical velocity of the atmospheric wind and the ocean $\watm$ and $\wocn$, respectively.
This can be summarized by
\begin{align}
    \ddnGau \vatm &= 0 \on{\Gau}, \quad \ddnGai \vatm = 0 \on{\Gai}, \quad \trGab \vocn = 0 \on{\Gab}, \text{ as well as}\label{eq:boundary conditions v}\\
    \trGau \watm &= 0 \on{\Gau}, \quad \trGai \watm = 0 \on{\Gai}, \quad \trGao \wocn = 0 \on{\Gao} \qandq \trGab \wocn = 0 \on{\Gab}.\label{eq:boundary conditions w}
\end{align}
On the lateral boundaries, we assume that all variables are periodic, i.\ e., $\vatm$ and $\patm$ periodic on $\Galatm$, $\vocn$ and $\pocn$ periodic on $\Galocn$ and $\vice$, $h$ and $a$ periodic on $\partial G$.
The coupling condition \eqref{eq:coupling} is assumed for the horizontal velocity of the ocean on the interface $\Gao$, and it is already included in \eqref{eq:coupled system}.

We also consider initial values $\vatm(0) = v_{\atm,0}$, $\vocn(0) = v_{\ocn,0}$, $\vice(0) = v_{\ice,0}$, $h(0) = h_0$ and $a(0) = a_0$.
The ground space is given by
\begin{equation*}
    \sX_0 := \rLqatm \times \rLqocn \times \rLq(G)^2 \times \rLq(G) \times \rLq(G).
\end{equation*}
Equipped with a product norm, $\sX_0$ becomes a Banach space.
Moreover, we define
\begin{equation*}
    \begin{aligned}
        \sX_1 := \{ &(\vatm,\vocn,\vice,h,a) \in (\rHsqper{2}(\Omegaatm)^2 \cap \rLqsigmabar(\Omegaatm)) \times \rLqsigmabar(\Omegaocn) \times \rHsqper{2}(G)^2 \times \rHsqper{2}(G) \times \rHsqper{2}(G) :\\
        &\Pocn \Delta \vocn \in \rLqsigmabar(\Omegaocn), \enspace \trGao \vocn = \vice \on{\Gao}, \text{ and \eqref{eq:boundary conditions v} is satisfied}\},
    \end{aligned}
\end{equation*}
where $\Pocn$ denotes the hydrostatic Helmholtz projection associated to the ocean as made precise in \autoref{ssec:hydrostatic Helmholtz and Stokes}, and $\watm$ and $\wocn$ can be recovered from $\vatm$ and $\vocn$, respectively, see also \autoref{ssec:hydrostatic Helmholtz and Stokes}.
For the initial data, we introduce the trace space $\sX_{\gamma} = (\sX_0,\sX_1)_{1-\nicefrac{1}{p},p}$, which is discussed in more detail in \autoref{sec:decoupling}.
For $p$, $q \in (1,\infty)$ with $\nicefrac{1}{p} + \nicefrac{1}{q} < \nicefrac{1}{2}$, we have $u = (\vatm,\vocn,\vice,h,a) \in \sX_{\gamma}$ if and only if 
\begin{equation*}
    \begin{aligned}
        &\vatm \in \rB_{qp,\per}^{2-\nicefrac{2}{p}}(\Omegaatm)^2 \cap \rL_{\sigmabar}^q(\Omegaatm), \enspace \vocn \in \rB_{qp,\per}^{2-\nicefrac{2}{p}}(\Omegaocn)^2 \cap \rL_{\sigmabar}^q(\Omegaocn) \text{ and } (\vice,h,a) \in \rB_{qp,\per}^{2-\nicefrac{2}{p}}(G)^4, \text{ with}\\
        &\ddnGau \vatm = \ddnGai \vatm = 0, \enspace \trGab \vocn = 0 \text{ and } \trGao \vocn = \vice.
    \end{aligned}
\end{equation*}

We introduce an open set $V = \Vatm \times \Vocn \times \Vice \subset \sX_{\gamma}$ such that for $(\vatm,\vocn,\vice,h,a) \in V$, we have
\begin{equation}\label{eq:shape V}
    \kone < h < \ktwo \quad \text{for} \quad 0 < \kone < \ktwo < \infty \quad \qandq a \in (0,1).
\end{equation}
Below, we provide conditions on the initial data and the external forcing terms for the main results.
\begin{asu}\label{ass:initial data}
For $V \subset \sX_{\gamma}$ from \eqref{eq:shape V}, the initial data satisfy $u_0 \in V$, so 
\begin{equation*}
    u_0 = (v_{\atm,0}, v_{\ocn,0}, v_{\ice,0}, h_0, a_0) \in \rB_{qp,\per}^{2-\nicefrac{2}{p}}(\Omegaatm)^2 \cap \rL_{\sigmabar}^q(\Omegaatm) \times \rB_{qp,\per}^{2-\nicefrac{2}{p}}(\Omegaocn)^2 \cap \rL_{\sigmabar}^q(\Omegaocn) \times \rB_{qp,\per}^{2-\nicefrac{2}{p}}(G)^4
\end{equation*}
subject to \eqref{eq:boundary conditions v} and $\trGao \vocn = \vice$, and so that $h_0 \in (\kappa_1,\kappa_2)$ and $a_0 \in (0,1)$.
Besides, for $T > 0$, the external forcing terms fulfill $\fatm \in \rLp(0,T;\rL^q(\Omegaatm)^2)$, $\focn \in \rLp(0,T;\rL^q(\Omegaocn)^2)$ and $\nablaH H \in \rLp(0,T;\rLq(G)^2)$.
\end{asu}
Our first main result on the local existence and uniqueness of a strong solution to \eqref{eq:coupled system} reads as follows.

\begin{thm}\label{thm:main thm}
	Let $p,q \in (1,\infty)$ such that $\nicefrac{1}{p} + \nicefrac{1}{q} < \nicefrac{1}{2}$, and assume that the initial data $u_0$ and the external forcing terms $\fatm$, $\focn$ and $\nablaH H$ satisfy \autoref{ass:initial data}.
	Then there exist $T' = T'(u_0) \in (0,T]$ and $r = r(u_0)$ with $\oB_{\sX_{\gamma}}(u_0,r) \subset V$ such that for for each initial value $u_1 \in \oB_{\sX_{\gamma}}(u_0,r)$, the coupled system~\eqref{eq:coupled system}, subject to the boundary conditions \eqref{eq:boundary conditions v} and \eqref{eq:boundary conditions w}, admits a unique strong solution
	\begin{equation*}
	    u(\cdot,u_1) = (\vatm, \vocn,\vice,h,a)(\cdot,u_1) \in \E_{T'} := \rWsp{1}(0,T';\sX_0) \cap \rLp(0,T';\sX_1).
	\end{equation*}
\end{thm}

We briefly discuss some further properties of the solution from \autoref{thm:main thm}.

\begin{cor}\label{cor:smoothing}
    \begin{enumerate}[(a)]
\item The solution class satisfies $\rWsp{1}(0,T';\sX_0) \cap \rLp(0,T';\sX_1) \hra \rC([0,T'],\sX_{\gamma})$.
\item Given an initial value $u_0 \in V$, there exists $C = C(u_0)$ such that for all $u_1, u_2 \in \oB_{\sX_{\gamma}}(u_0,r)$, it holds that 
$\| u(\cdot,u_1) - u(\cdot,u_2) \|_{\E_{T'}} \le C \| u_1 - u_2 \|_{\sX_\gamma}$.
\item The solution regularizes instantly in time, i.\ e., $t \dk{t}v \in \rWsp{1}(0,T';\sX_0) \cap \rLp(0,T';\sX_1)$. In particular, $u \in \rC^1([b,T'];\sX_\gamma) \cap \rC^{1- \nicefrac{1}{p}}([b,T'];\sX_1)$ 
for every $b \in (0,T')$.
\item The solution $u = u(u_0)$ exists on a maximal time interval $J(u_0) = [0,t_+(u_0))$, and the latter is characterized by the alternatives \\
(i)~global existence, i.\ e., it holds that $t_+(u_0) = \infty$, or \\
(ii)~$\lim_{t \to t_+(u_0)} \mathrm{dist}(u(t),\partial V) = 0$, or \\
(iii)~$\lim_{t \to t_+(u_0)} u(t)$ does not exist in $\sX_\gamma$.
\end{enumerate}
\end{cor}

For $h_*$ and $a_*$ constant in space and time, it readily follows that $(0,0,0,h_*,a_*)$ are equilibria to \eqref{eq:coupled system} provided the external forces $\fatm$, $\focn$ as well as $g \nablaH H$  vanish and the thermodynamic terms $\Sh$ and $\Sa$ are neglected.
Setting $\tatm = 0$, we then verify the normal stability of the above equilibria in the trace space $\sX_\gamma$, resulting in the existence of a unique global strong solution for initial data close to the equilibrium $(0,0,0,h_*,a_*)$.
This leads to the second main result of this article.

\begin{thm}\label{thm:second main thm}
For $h_* \in (\kappa_1,\kappa_2)$ and $a_* \in (0,1)$ constant in space and time, $u_* = (0,0,0,h_*,a_*)$ is a stable equilibrium in $\sX_\gamma$.
Moreover, there exists $r > 0$ such that the unique solution $u$ of \eqref{eq:coupled system} with $\fatm = \focn = g \nablaH H = \tatm = 0$, $\Sh = \Sa = 0$ and $u_0$ fulfilling \autoref{ass:initial data} as well as $\| u_0 - u_* \|_{\sX_\gamma} < r$ exists on $\R_+$ and converges at an exponential rate in $\sX_\gamma$ to some equilibrium $u_\infty$ of \eqref{eq:coupled system} as $t \to \infty$. 
\end{thm}

The following remarks on \autoref{thm:main thm} and \autoref{thm:second main thm} are in order at this point. 

\begin{rem}\label{cor:possible extenions main result} {\rm 
For simplicity of the presentation we did not introduce time weights $\mu \in (\nicefrac{1}{p},1]$ for the solution space to exploit the parabolic regularization.
The corresponding space for the initial data then becomes $\sX_{\gamma,\mu} = (\sX_0,\sX_1)_{\mu-\nicefrac{1}{p},p}$.

Assuming that for $p,q \in (1,\infty)$ such that $\nicefrac{1}{2} + \nicefrac{1}{p} + \nicefrac{1}{q} < \mu \le 1$, the initial data $u_0$ lie in an open set $V_\mu \subset \sX_{\gamma,\mu}$ as 
in \eqref{eq:shape V}, and supposing that the external forcing terms satisfy the above assumptions, we obtain an analogous statement as in \autoref{thm:main thm}, but the solution $u$ is contained 
in the class $u \in \rWsp{1}_\mu(0,T';\sX_0) \cap \rLp_\mu(0,T';\sX_1)$. }
\end{rem}

\begin{rem} {\rm (a) The statements of \autoref{thm:main thm} and \autoref{thm:second main thm} remain valid when including Coriolis terms as in \cite[Section~2]{BDHH:22}.
They are omitted here for simplicity of the presentation. 

(b) Let $\ReHatm$, $\ReHocn$, $\Rezatm$ and $\Rezocn$ denote the horizontal and vertical Reynolds numbers in the context of the atmosphere and the ocean.
Instead of $\Delta \vatm$ and $\Delta \vocn$, one can also consider $\nicefrac{1}{\ReHatm} \DeltaH \vatm + \nicefrac{1}{\Rezatm} \partial_z^2 \vatm$ and 
$\nicefrac{1}{\ReHocn} \DeltaH \vocn + \nicefrac{1}{\Rezocn} \partial_z^2 \vocn$ in \eqref{eq:coupled system}.
The statements of \autoref{thm:main thm} and \autoref{thm:second main thm} remain valid in this case.}
 \end{rem}

\section{Rewriting the Coupled System as a Quasilinear Evolution Equation}\label{sec:coupled system as qee}

The first step in the proof of \autoref{thm:main thm} is to write the coupled system of PDEs \eqref{eq:coupled system} as a quasilinear evolution equation in the ground space $\sX_0$.

\subsection{The hydrostatic Helmholtz projection and the hydrostatic Stokes operator}\label{ssec:hydrostatic Helmholtz and Stokes}

\

For $a,b \in \R$ with $a < b$, we consider the cylindrical domain $\Omega = G \times (a,b)$. 
The incompressible, isothermal primitive equations for the velocity $u = (v,w)$ of the fluid and the surface pressure $\pi$ are given by
\begin{equation}
	\left\{
	\begin{aligned}
		\dt v - \Delta v + \bilinu{u}{v} + \nablaH \pi &= f, &&\onOmegaT, \\
		\dz  \pi &= 0, &&\onOmegaT, \\
		\div u &= 0, &&\onOmegaT, \\
		v(0) &= u_0, &&\onOmega . 
	\end{aligned}
	\right.
	\label{eq:primitive equation general}
\end{equation} 
Here $f$ is an external force term, and $v = (v_1,v_2)$ denotes the horizontal velocity, whereas $w$ represents the vertical velocity. 
In the sequel, the remaining surface pressure will be denoted by $\pi_s$.
On the upper and bottom parts of the boundary $\partial \Omega$, $v$ satisfies Dirichlet or Neumann conditions conditions on $\GaD$ and $\GaN$, respectively, while $w$ is subject to Dirichlet boundary conditions.
Both, $v$ and $w$, as well as the pressure $\pi_s$ are assumed to be periodic on the lateral boundary $\Gal$.

In conjunction with the boundary conditions $w = 0$ on $G \times \{a,b\}$, the incompressibility condition $\div u = 0$ yields that the vertical component $w$ can be expressed by $w(\cdot,\cdot,z) = -\int_a^z \divH v(\cdot,\cdot,\xi) \d \xi$, $z \in [a,b]$.

Denoting by $\vbar$ the vertical average of $v$, the fracturing part is given by $\vtilde := v - \vbar$.
The vertical average commutes with horizontal or tangential derivation and with the horizontal Laplacian for sufficiently smooth vector fields $v$, i.\ e., $\bar{\nablaH v} = \nablaH \vbar$ and $\bar{\DeltaH v} = \DeltaH \vbar$.
In particular, the incompressibility condition can be expressed as $\divH \vbar = 0$.

We now introduce the associated hydrostatic Helmholtz or Leray projection $\P \colon \rLq(\Omega)^2 \to \rLqOmegasigmabar$.
It is a bounded, linear projection which annihilates the pressure term, i.\ e., $\P (\nablaH \pi) = 0$, and the space $\rLq(\Omega)$ can be decomposed into $\rLq(\Omega)^2 = \rLqOmegasigmabar \oplus \{ \nablaH \pi \colon \pi \in \widehat{\rW}^{1,q}(G) \}$ along the hydrostatic Helmholtz projection $\P$. 
Furthermore, the Helmholtz projection commutes with the time derivative, i.\ e., $\dt \P = \P \dt$.
It is related to the two-dimensional classical Helmholtz projection~$\PH$  by
\begin{equation*}
	\P \colon \rLq(\Omega)^2 \to \rLqOmegasigmabar, \quad \sP v = \PH \vbar + \vtilde .
\end{equation*}
We define the maximal hydrostatic Stokes operator $\Am \colon D(\Am) \subset \rLqOmegasigmabar \to \rLqOmegasigmabar$ by $\Am v := \P \Delta v$, where $D(\Am) := \{v \in \rLqOmegasigmabar \colon \Am v \in \rLqOmegasigmabar\}$.
The hydrostatic Stokes operator with boundary conditions $\Abc \colon D(\Abc) \subset \rLqOmegasigmabar \to \rLqOmegasigmabar$ is defined by $\Abc v := \P \Delta v$, with $D(\Abc) := \rHsqbcper{2}(\Omega)^2 \cap \rLqOmegasigmabar$.

For $s \in (0,2]$ and $\GaD$ as well as $\GaN$ representing the part of the boundary with Dirichlet or Neumann boundary conditions, respectively, we set
\begin{equation*}
	\rHsqbcper{s}(\Omega)
	:= 
	\left\{
	\begin{aligned}
		&\left\{ v \in \rHsqper{s}(\Omega)^2 \colon v|_{\GaD} = 0, \, \dz v|_{\GaN} = 0 \right\}, &&\for 1 + \nicefrac{1}{q} < s \leq 2, \\
		&\left\{ v \in \rHsqper{s}(\Omega)^2 \colon v|_{\GaD} = 0 \right\}, &&\for \nicefrac{1}{q} < s < 1 + \nicefrac{1}{q}, \\
		&\rHsqper{s}(\Omega)^2,
		&&\for 0 < s < \nicefrac{1}{q}.
	\end{aligned}
	\right.
\end{equation*}
Note that $D(\Am) \subset \rHsqbcper{s}(\Omega)$ only for $s < \nicefrac{1}{q}$ but not for $s > \nicefrac{1}{q}$.
As indicated above, $w$ can be recovered from $v$, so the bilinearity can be written as $\bilinu{u}{v} = \bilin{v}{v}$.
Applying the hydrostatic Helmholtz projection to the bilinearity, we obtain
\begin{equation}
	\rF(v,v') := \P(\bilin{v}{v'})
	= \P(\bilinH{v}{v'}) + \P(w(v)\dz v'), \text{ with } \rF(v) := \rF(v,v).
	\label{eq:bilinearity primitive eq}
\end{equation}

\subsection{Hibler's sea ice operator}
\label{ssec:Hibler operator}

\

Proceeding as in \cite[Section~3]{BDHH:22}, we introduce the map $\S \colon \R^{2 \times 2} \to \R^{2 \times 2}$ with
\begin{equation*}
    \S \eps = \begin{pmatrix}
        (1 + \frac{1}{e^2}) \eps_{11} + (1 - \frac{1}{e^2}) \eps_{22} & \frac{1}{e^2}(\eps_{12} + \eps_{21})\\
        \frac{1}{e^2}(\eps_{12} + \eps_{21}) & (1 - \frac{1}{e^2}) \eps_{11} + (1 + \frac{1}{e^2}) \eps_{22}
    \end{pmatrix}.
\end{equation*}
Identifying $\eps \in \R^{2 \times 2}$ with the vector $(\eps_{11},\eps_{12},\eps_{21},\eps_{22})^\top$, we find that the map $\mathbb{S}$ corresponds to the positive semi-definite matrix
\begin{equation*}
    \mathbb{S} = \left(\mathbb{S}_{ij}^{kl}\right) = \begin{pmatrix}
    1 + \frac{1}{e^2} & 0 & 0 & 1 - \frac{1}{e^2}\\
    0 & \frac{1}{e^2} & \frac{1}{e^2} & 0\\
    0 & \frac{1}{e^2} & \frac{1}{e^2} & 0\\
    1 - \frac{1}{e^2} & 0 & 0 & 1 + \frac{1}{e^2}
    \end{pmatrix}.
\end{equation*}
Setting $\Sd = \Sd(\eps,P) := \frac{P}{2} \frac{\S \eps}{\trid(\eps)}$, we define Hibler's operator by $\sAH \vice := \frac{1}{\rice h}\divH \Sd(\vice)$ and recall
\begin{equation*}
    (\sAH \vice)_i = -\sum_{j,k,l=1}^2 \frac{P}{2 \rice h} \frac{1}{\trid(\eps)} \left(\S_{ij}^{kl} - \frac{1}{\trid^2(\eps)}(\S \eps)_{ik} (\S \eps)_{jl}\right) \D_k \D_l v_{\ice,j} + \frac{1}{2 \rice h \trid(\eps)} \sum_{j=1}^2 (\dk{j}P)(\S \eps)_{ij}
\end{equation*}
from \cite[Section~3]{BDHH:22}, where $i=1,2$ and $\D_m = - \i \dk{m}$.
The principal part of $\sAH$ has coefficients given by
\begin{equation*}
    a_{ij}^{kl}(\eps,h,a) := -\frac{P(h,a)}{2 \rice h} \frac{1}{\trid(\eps)} \left(\S_{ij}^{kl} - \frac{1}{\trid^2(\eps)}(\S \eps)_{ik} (\S \eps)_{jl}\right).
\end{equation*}
For sufficiently smooth initial data $v_0 = (v_{\ice,0},h_0,a_0)$, the linearization of Hibler's operator is of the form
\begin{equation*}
    [\sAH(v_0)\vice]_i = \sum_{j,k,l=1}^2 a_{ij}^{kl}(\eps_0,h_0,a_0) \D_k \D_l v_{\ice,j} + \frac{1}{2 \rice h_0 \trid(\eps_0)} \sum_{j=1}^2 (\dk{j}P(h_0,a_0))(\S \eps(\vice))_{ij}.
\end{equation*}

The $\rLq$-realization $\AH(v_0)$ of $\sAH(v_0)$ on $G = (0,1) \times (0,1)$ is then defined by $[\AH(v_0)] \vice := [\sAH(v_0)]\vice$ for $\vice \in D(\AH(v_0)) = \rHsqper{2}(G)^2$.

Comparing this definition of Hibler's operator to the one presented in \cite[Section~3]{BDHH:22}, we observe that the sign is changed and $\nicefrac{1}{\rice h}$ is included in the definition for the sake of consistency with the hydrostatic Stokes operator and to ease notation.
We also introduce the lower-order terms
\begin{equation*}
    \Bh(h_0,a_0) h := - \frac{\dk{h}P(h_0,a_0)}{2 \rice h_0} \nablaH h \quad \andtext \quad  \Ba(h_0,a_0) a := - \frac{\dk{a}P(h_0,a_0)}{2 \rice h_0} \nablaH a,
\end{equation*}
originating from $\frac{1}{m}\divH \frac{P}{2}\I$.
Finally, we recall that $\DeltaH$ is the horizontal Laplacian on the square $G$ with periodic boundary conditions, and its domain is given by $D(\DeltaH) = \rHsqper{2}(G)$.

\subsection{A new formulation of the coupled system}
\label{ssec:new formulation}

\
 
We start the reformulation of \eqref{eq:coupled system} with the primitive equations of the atmosphere. 
To this end, we choose $a = \ktwo$ and $b = \hatm$ and denote the atmospheric solenoidal $\rLq$-space by $\rLqatm$.
The atmospheric Helmholtz projection $\Patm$ and the atmospheric Stokes operator $\Aatm \colon D(\Aatm) \subset \rLqatm \to \rLqatm$ are defined as in \autoref{ssec:hydrostatic Helmholtz and Stokes}, i.\ e.,
\begin{equation*}
    \left\{
    \begin{aligned}
        \Aatm \vatm &= \Patm \Delta \vatm,\\
        D(\Aatm) &= \left\{\vatm \in \rHsqper{2}(\Omegaatm)^2 \cap \rLqsigmabar(\Omegaatm) : \ddnGau \vatm = 0 \on{\Gau} \andtext \ddnGai \vatm = 0 \on{\Gai}\right\}
    \end{aligned}
    \right.
\end{equation*}
in view of \eqref{eq:boundary conditions v}, and the bilinearity $\Fatm$ is defined via \eqref{eq:bilinearity primitive eq}. 

For the primitive equations of the ocean, we choose $a = -\hocn$ and $b = 0$. 
The oceanic solenoidal $\rLq$-space is denoted by $\rLqocn$, and we define the oceanic Helmholtz projection $\Pocn$ as well as the maximal oceanic Stokes operator $\Amocn \colon D(\Amocn) \subset \rLqocn \to \rLqocn$ as in \autoref{ssec:hydrostatic Helmholtz and Stokes}, so
\begin{equation*}
    \Amocn \vocn = \Pocn \Delta \vocn, \enspace D(\Amocn) = \{\vocn \in \rLqsigmabar(\Omegaocn) : \Pocn \Delta \vocn \in \rLqsigmabar(\Omegaocn)\},
\end{equation*}
and the bilinearity $\Focn$ via \eqref{eq:bilinearity primitive eq}. 
The oceanic hydrostatic Stokes operator with homogeneous boundary conditions $\An \colon D(\An) \subset \rLqocn \to \rLqocn$ is given by\begin{equation}\label{eq:boundary conditions A0}
    \left\{
    \begin{aligned}
    \An \vocn &= \Pocn \Delta \vocn,\\
	D(\An) &= \left\{\vocn \in \rHsqper{2}(\Omegaocn)^2 \cap \rLqsigmabar(\Omegaocn) : \trGab \vocn = 0 \on{\Gab} \andtext \trGao \vocn = 0 \on{\Gao}\right\}
	\end{aligned}
    \right.
\end{equation}
with regard to \eqref{eq:boundary conditions v} and \eqref{eq:coupling}.
Setting $\Coi(h) := \nicefrac{\muocn}{\rice h}$ so that $\nicefrac{1}{\rice h} \tocn = -\Coi(h) \ddnGao \vocn$, for $(v_{\ice,0},h_0,a_0) \in \Vice$, we define $\sA(v_{\ice,0},a_0,h_0) \colon \sX_1 \subset \sX_0 \to \sX_0$ by
\begin{equation}
	\begin{aligned}
		\sA(v_{\ice,0},h_0,a_0) 
		&:= 
		\begin{pmatrix}
			\Aatm & 0 & 0 & 0 & 0 \\
			0 & \Amocn & 0 & 0 & 0 \\
			0 & -\Coi(h_0) \ddnGao & \AH(v_{\ice,0},h_0,a_0) & \Bh(h_0,a_0) & \Ba(h_0,a_0) \\
			0 & 0 & 0 & \Deltah & 0 \\
			0 & 0 & 0 & 0 & \Deltaa 
		\end{pmatrix},
		\\
		\sX_1 := D(\sA) &:= 
		\{ (\vatm,\vocn,\vice,h,a) \in D(\Aatm) \times D(\Amocn) \times D(\AH) \times D(\DeltaH) \times D(\DeltaH) \colon \\
		&\qquad\trGab \vocn = 0 \on{\Gab} \andtext \trGao \vocn = \vice \on{\Gao}\}.
	\end{aligned}
	\label{eq:operator matrix}
\end{equation}
Note that this operator matrix has {\em non-diagonal domain}. 
Further, we define the nonlinearity by
\begin{equation}
		\sF(\vatm,\vocn,\vice,h,a) := 
		\begin{pmatrix}
			\Patm(\bilinatm)  \\
			\Pocn(\bilinocn)  \\
			\bilinHice - \frac{1}{\rice h}\tatm\\
			\divH(\vice h) - \Sh \\
			\divH(\vice a) - \Sa
		\end{pmatrix},
	\label{eq:nonlinearity}
\end{equation}
with $\Sh$ and $\Sa$ as in \eqref{eq:thermodynamic terms}. 
Besides, the external force is given by $f := (\Patm\fatm,\Pocn\focn,- g \nablaH H,0,0)^\top$.

The coupled system \eqref{eq:coupled system} subject to the boundary conditions \eqref{eq:boundary conditions v} and \eqref{eq:boundary conditions w} can now be reformulated as an abstract quasilinear Cauchy problem
\begin{equation}
	\left\{
	\begin{aligned} 
		\dt u - \sA(u) u + \sF(u) &= f, &&t \in [0,T], \\
		u(0) &= u_0,
	\end{aligned}
	\right. 
	\label{eq:quasilinear Cauchy problem} 
\end{equation}
on the ground space $\sX_0$.

\section{Hibler's Sea Ice Equations on the Square with Periodic Boundary Conditions}
\label{sec:Hibler}

In the first part of this section, we show that Hibler's operator admits a bounded $\Hinfty$-calculus, while the second part is concerned with suitable Lipschitz estimates.

\subsection{Bounded $\Hinfty$-calculus of Hibler's operator}
\label{ssec:MR Hibler}

\

The embedding below is classical, see e.\ g.\ \cite[Theorem~4.6.1]{Tri:78}.   Let $p,q \in (1,\infty)$ with $\nicefrac{1}{p} + \nicefrac{1}{q} < \nicefrac{1}{2}$. Then 
\begin{equation}\label{lem:embedding besov space}
\rBsqp{2-\nicefrac{2}{p}}(G) \hra \rC^1(\overline{G}).
\end{equation}
Making use of \eqref{lem:embedding besov space} and proceeding as in \cite[Section~4]{BDHH:22}, we obtain the ellipticity of $-\sAH(v_0)$.
For the notion of strong ellipticity and parameter ellipticity, we refer e.\ g.\ to \cite[Definition~5.1]{DHP:03}.

\begin{lem}\label{lem:parameter-ellipticity}
    Let $p,q \in (1,\infty)$ such that $\nicefrac{1}{p} + \nicefrac{1}{q} < \nicefrac{1}{2}$, and let $u_0 = (v_{\ice,0},h_0,a_0) \in \Vice$.
    Then the principal part of the negative Hibler operator $-\sAH(v_0)$ is strongly elliptic and also parameter-elliptic of angle $\phi_{-\sAH(v_0)} = 0$.
\end{lem}

We will now derive the bounded $\Hinfty$-calculus for the linearized $\rLq$-realization $-\AH(v_0)$.

\begin{prop}\label{prop:MR and Hinfty AH}
	Let $p,q,r,s \in (1,\infty)$ such that $\nicefrac{1}{p} + \nicefrac{1}{q} < \nicefrac{1}{2}$, let $v_0 = (v_{\ice,0},h_0,a_0) \in \Vice$, and let $\AH(v_0)$ denote the $\rLq$-realization of the linearized Hibler operator.
	Then there is $\omega_0 \in \R$ such that for all $\omega > \omega_0$, $-\AH(v_0) + \omega$ admits a bounded $\Hinfty$-calculus on $\rL^r(G)^2$ with angle $\Phi_{-\AH(v_0) + \omega}^\infty < \nicefrac{\pi}{2}$.
\end{prop}

\begin{proof}
    Identifying functions on the square $G = (0,1) \times (0,1)$ endowed with periodic boundary conditions with functions on the torus $\mathbb{T}$, exploiting 
\autoref{lem:parameter-ellipticity} for the ellipticity properties of the negative Hibler operator $-\sAH(v_0)$ and using \eqref{lem:embedding besov space} in conjunction with the fact 
that the coefficients of $\sAH(v_0)$ depend smoothly on $\nablaH v_{\ice,0}$, $\nablaH h_0$, $\nablaH a_0$, $h_0$ and $a_0$, we conclude from \cite[Theorem~7.1]{DS:97} that there is $\omega_0 \in \R$ such that for all $\omega > \omega_0$, the operator $-\AH(v_0) + \omega$ admits a bounded $\Hinfty$-calculus on $\rL^r(\mathbb{T})^2 = \rL^r(G)^2$ with angle $\Phi_{-\AH(v_0) + \omega}^\infty < \nicefrac{\pi}{2}$.
\end{proof}

The following result is a direct consequence of \autoref{prop:MR and Hinfty AH}, because the latter proposition implies that $-\AH(v_0) + \omega$ has bounded imaginary powers on $\rL^q(G)^2$, see also the relation \eqref{eq:rel Hinfty BIP RS}. 

\begin{cor}\label{cor:Bh Ba relatively AH-bd}
The linearized operators $\Bh$ and $\Ba$ are relatively $(-\AH+\omega)^{\frac{1}{2}}$-bounded.
\end{cor}

\subsection{Estimates of the nonlinearities}
\label{ssec:nonlinearity sea ice}

\

We now show that the nonlinear terms in the sea ice equations satisfy certain Lipschitz estimates.
To this end, we set
\begin{equation*}
    \Aice(v_0) = \begin{pmatrix}
        0 & -\Coi(h_0) \ddnGao & \AH(v_0) & \Bh(h_0,a_0) & \Ba(h_0,a_0)\\
        0 & 0 & 0 & \Deltah & 0\\
        0 & 0 & 0 & 0 & \Deltaa
    \end{pmatrix}
\end{equation*}
as well as 
\begin{equation*}
    \Fice(v) = \begin{pmatrix}
        \bilinHice - \frac{1}{\rice h}\tatm(\vatm)\\
			\divH(\vice h) - \Sh\\
			\divH(\vice a) - \Sa
    \end{pmatrix}.
\end{equation*}

Further, with regard to \autoref{prop:MR and Hinfty AH}, it is natural to introduce
\begin{equation*}
    \Aiceom(v_0) = \begin{pmatrix}
        0 & -\Coi(h_0) \ddnGao & \AH(v_0) - \omega & \Bh(h_0,a_0) & \Ba(h_0,a_0) \\
        0 & 0 & 0 & \Deltah & 0\\
        0 & 0 & 0 & 0 & \Deltaa
    \end{pmatrix}
\end{equation*}
and $\Ficeom(v) = \Fice(v) + (\omega \vice,0,0)^\top$.
The thermodynamic terms $\Sh$ and $\Sa$ are defined by
\begin{equation}\label{eq:thermodynamic terms}
    \Sh = f_1(\nicefrac{h}{a})a + (1-a) f_1(0) \qandq \Sa = \begin{cases*} \frac{f_1(0)}{\kone}(1-a),& if $f_1(0) > 0$, \\ 0, \quad & if $f_1(0) < 0$, 
    \end{cases*}
    \quad + \quad \begin{cases*} 0,& if $\Sh > 0$, \\ \frac{a}{2 h}\Sh, & if $\Sh < 0$,
    \end{cases*}
\end{equation}
where $f_1 \in \rC_b^1([0,\infty);\R)$ is an arbitrary function representing the ice growth rate, see e.\ g.\ the one considered by Hibler \cite{Hib:79}.
Moreover, for $i \in \{0,1,\theta,\gamma,\beta\}$, we introduce $\sY_i = \sY_i^\atm \times \sY_i^\ocn \times \sY_i^\ice$ given by
\begin{equation}
    \left\{
    \begin{aligned}
        \sYzero &= \rLqatm \times \rLqocn \times \rLq(G)^4,\\
        \sYone &= \rHsqper{2}(\Omegaatm)^2 \cap \rLqatm \times \rHsqper{2}(\Omegaocn)^2 \cap \rLqocn \times \rHsqper{2}(G)^4,\\
        \sYtheta &= \rBsqpper{2\theta}(\Omegaatm)^2 \cap \rLqatm \times \rBsqpper{2\theta}(\Omegaocn)^2 \cap \rLqocn \times \rBsqpper{2\theta}(G)^4,\\
        \sYgamma &= \sYthetaa \times \sYthetao \times \sYthetai \text{ for } \theta = 1 - \nicefrac{1}{p}, \andtext\\
        \sYbeta &= \rHsqbcper{2\beta}(\Omegaatm)^2 \cap \rLqatm \times \rHsqper{2\beta}(\Omegaocn)^2 \cap \rLqocn \times \rHsqper{2\beta}(G)^4.
    \end{aligned}
    \right.
    \label{eq:Y spaces}
\end{equation}

In addition, for $u \in \sYgamma$ and $r>0$, we use $\oB_{\sYgamma}(u,r)$ to denote the closed ball of center $u$ with radius $r$ in~$\sYgamma$.
Let further $W = \Watm \times \Wocn \times \Wice \subset \sYgamma$ be an open subset such that for $u_0 \in W$, the properties from \eqref{eq:shape V} are satisfied. 
We then get the following Lipschitz estimates.

\begin{lem}\label{lem:estimates rhs sea ice}
	Let $p,q \in (1,\infty)$ such that $\nicefrac{1}{p} + \nicefrac{1}{q} < \nicefrac{1}{2}$, and let $u_0 = (v_{\atm,0}, v_{\ocn,0}, v_0) \in W$.
	Then there exists $r_0 > 0$ and a constant $L > 0$ such that $\oB_{\sYgamma}(u_0,r_0) \subset W$, and for all $v_1, v_2 \in \sYgammai$ with $u_1, u_2 \in \oB_{\sYgamma}(u_0,r_0)$ and $u \in \sYone$, it holds that $\| \Aiceom(v_1)u - \Aiceom(v_2)u \|_{\sYzeroi} \le L \|v_1 - v_2\|_{\sYgammai} \|u\|_{\sYone}$ as well as $\| \Ficeom(v_1) - \Ficeom(v_2) \|_{\sYzeroi} \le L \| v_1 - v_2 \|_{\sYgammai}$.
\end{lem}

\begin{proof}
    First, choose $r_0 > 0$ sufficiently small such that $\oB_{\sYgamma}(u_0,r_0) \subset W$ thanks to $W$ being open, and consider $v_1, v_2$ with $u_1, u_2 \in \oB_{\sYgamma}(u_0,r_0)$ as well as
    \begin{equation*}
        u = (\vatm,\vocn,\vice,h,a) \in \sYone = \rHsqper{2}(\Omegaatm)^2 \cap \rLqatm \times \rHsqper{2}(\Omegaocn)^2 \cap \rLqocn \times \rHsqper{2}(G)^4.
    \end{equation*}

As the spaces with horizontal periodicity embed continuously into the ones without, we omit the corresponding subscript in the remainder of the proof for convenience.
For most terms, we refer to \cite[Lemma~6.2]{BDHH:22}, because the procedure from this lemma carries over to the present situation.
Since the velocity of the atmospheric wind now acts as a variable, we verify the respective Lipschitz estimate separately.
To this end, we first observe that for suitable $g_1$, $g_2$, it holds that
\begin{equation*}
    \| |g_1| g_1 - |g_2| g_2 \|_{\rLq(G)} \le \| g_1 \|_{\rL^{2q}(G)} \| g_1 - g_2 \|_{\rL^{2q}(G)} + \| g_1 - g_2 \|_{\rL^{2q}(G)} \| g_2 \|_{\rL^{2q}(G)}.
\end{equation*}
Thanks to $\nicefrac{1}{p} + \nicefrac{1}{q} < \nicefrac{1}{2}$, there is $\eta > 0$ with $2 - \nicefrac{2}{p} - \eta - \nicefrac{2}{q} \ge 1 - \nicefrac{2}{2q}$, so \cite[Theorem~4.6.1]{Tri:78} yields that
\begin{equation}\label{eq:embedding Besov into Bessel}
     \rBsqp{2-\nicefrac{2}{p}}(G) \hra \rB^{2 - \nicefrac{2}{p} - \eta}_{q 2q}(G) \hra \rH^{1,2q}(G), \enspace \rBsqp{2 - \nicefrac{2}{p}}(\Omegaatm) \hra \rB^{2 - \nicefrac{2}{p} - \eta}_{q 2q}(\Omegaatm) \hra \rW^{\nicefrac{1}{q}+\nicefrac{\eta}{2},2q}(\Omegaatm).
\end{equation}
    
Thus, using the latter embeddings together with $v_i \in \Wice$ and continuity of $\trGai$ as an operator from $\rW^{\nicefrac{1}{q}+\nicefrac{\eta}{2},2q}(\Omegaatm)$ to $\rL^{2q}(G)$, we obtain
\begin{equation*}
\begin{aligned}
    &\quad \Big\| \frac{\ratm \Catm}{\rice h_1} (|\trGai v_{\atm,1}| \Ratm \trGai v_{\atm,1} - |\trGai v_{\atm,2}| \Ratm \trGai v_{\atm,2}) \Big\|_{\rLq(G)}\\
    &\le C\Bigl(\| v_{\atm,1} \|_{\rBsqp{2 - \nicefrac{2}{p}}(\Omegaatm)} + \| v_{\atm,2} \|_{\rBsqp{2 - \nicefrac{2}{p}}(\Omegaatm)}\Bigr) \| v_{\atm,1} - v_{\atm,2} \|_{\rBsqp{2 - \nicefrac{2}{p}}(\Omegaatm)}\\
    &\le C(r_0 + \| u_0 \|_{\sYgamma}) \| u_1 - u_2 \|_{\sYgamma}.
\end{aligned}
\end{equation*}

Similarly, we get
\begin{equation*}
\begin{aligned}
    &\Big \| \frac{1}{h_1} - \frac{1}{h_2} \Big \|_{\rL^\infty(G)} \cdot \| |\trGai v_{\atm,1}| \Ratm \trGai v_{\atm,1} \|_{\rLq(G)}
    \le \, &C(r_0 + \| u_0 \|_{\sYgamma})^2 \| u_1 - u_2 \|_{\sYgamma},
\end{aligned}
\end{equation*}
where we employed \eqref{eq:embedding Besov into Bessel}.
Therefore, a concatenation of the previous two estimates results in
\begin{equation*}
    \Big \| \frac{1}{\rice h_1} \tatm(v_{\atm,1}) - \frac{1}{\rice h_2} \tatm(v_{\atm,2}) \Big \|_{\rLq(G)} \le C \cdot \| u_1 - u_2 \|_{\sYgamma}
\end{equation*}
for a suitable constant $C$ depending on $r_0$ and $\| u_0 \|_{\sYgamma}$.
Recalling the shape of $\sYzero$ from \eqref{eq:Y spaces}, we conclude the Lipschitz estimate of $\Ficeom$.

Concerning $\Aiceom$, we refer to \cite[Lemma~6.2]{BDHH:22} except for $\Coi(h) \ddnGao$ which is new in comparison to the considerations in \cite{BDHH:22}.
The assumption $\nicefrac{1}{p} + \nicefrac{1}{q} < \nicefrac{1}{2}$ also implies continuity of $\ddnGao$ from $\rHsq{2}(\Omegaocn)$ to $\rLq(G)$.
A similar argument as above for the atmospheric wind terms exhibits that
\begin{equation*}
    \begin{aligned}
        \| \Coi(h_1) \ddnGao \vocn - \Coi(h_2) \ddnGao \vocn\|_{\rLq(G)} &\le \| \Coi(h_1) - \Coi(h_2) \|_{\rL^\infty(G)} \| \ddnGao \vocn \|_{\rLq(G)}\\
        &\le C \| h_1 - h_2 \| \| \vocn \|_{\rHsq{2}(\Omegaocn)}
        \le C \| v_1 - v_2 \|_{\sYgammai} \| u \|_{\sYone}.\qedhere
    \end{aligned}
\end{equation*}
\end{proof}

\section{Primitive Equations on Cylindrical Domains}
\label{sec:pe}

In this section, we collect properties of the primitive equations \eqref{eq:primitive equation general} on cylindrical domains $\Omega := G \times (a,b)$, where $G = (0,1) \times (0,1)$ and $-\infty < a < b < \infty$. 
Apart from recalling that the hydrostatic Stokes operator admits a bounded $\Hinfty$-calculus, we provide estimates of the bilinear term appearing in the primitive equations tailored to our setting.
Throughout this section, we denote by $v = (v_1,v_2)$ the principal variable corresponding to the horizontal velocity.

\subsection{Bounded $\Hinfty$-calculus of the hydrostatic Stokes operator}
\label{ssec:Hinfty hydrostatic Stokes}

\

We start by recalling the bounded $\Hinfty$-calculus of the hydrostatic Stokes operator from \autoref{ssec:hydrostatic Helmholtz and Stokes}. 

\begin{prop}[Theorem~3.1 in \cite{GGHHK:17}]\label{prop:bdd Hinfty calculus of hydrostatic Stokes}
    Let $q \in (1,\infty)$ and $\mu \ge 0$.
    Then $-\Abc + \mu$ has a bounded $\Hinfty$-calculus on $\rLqsigmabar(\Omega)$ with $\Phi_{-\Abc + \mu}^\infty = 0$ provided $\mu > 0$.
    In the case $\GaD \neq \emptyset$, the above assertion is also valid for $\mu = 0$.
\end{prop}

Choosing $a = \ktwo$, $b = \hatm$ and the boundary conditions \eqref{eq:boundary conditions v}, we obtain $\Abc = \Aatm$, while the choice $a = -\hocn$, $b = 0$ and the boundary conditions as in \eqref{eq:boundary conditions A0} lead to $\Abc = \An$.
Hence, we conclude the following from \autoref{prop:bdd Hinfty calculus of hydrostatic Stokes} and from results in \cite[Section~3]{GGHHK:17}.

\begin{cor}\label{cor:Hinfty Stokes atm ocn}
	Let $p,q \in (1,\infty)$ and $\mu \ge 0$. Then
	\begin{enumerate}[(a)]
		\item the operator $-\Aatm+\mu$ admits a bounded $\Hinfty$-calculus on $\rLqatm$ with $\Phi_{\Aatm}^\infty = 0$, and
		\item the operator $-\An$ admits a bounded $\Hinfty$-calculus on $\rLqocn$ with $\Phi_{\An}^\infty = 0$, and it holds that $0 \in \rho(-\An)$.
	\end{enumerate}
\end{cor}

\subsection{Estimates of the bilinearity}
\label{ssec:nonlinearity PE}

\

We recall the bilinearity $\rF$ from \eqref{eq:bilinearity primitive eq}.
As a preparation, we briefly introduce anisotropic function spaces, see also \cite[Section~5]{HK:16}.
For $s,t \ge 0$ and $1 \le p,q \le \infty$, we consider $\rWz{r}{q}\rWxy{s}{p} := \rWsq{r}((a,b);\rWsp{s}(G))$ equipped with the norms $\| v \|_{\rWz{r}{q}\rWxy{s}{p}} = \left\| \| v(\cdot,z) \|_{\rWsp{s}(G)} \right \|_{\rWsq{r}(a,b)}$, so they become Banach spaces.
The same remains valid when considering Bessel potential spaces instead of Sobolev spaces.
First, employing H\"older's inequality independently with respect to $z$ and $(x,y)$, we derive that for $p$, $p_1$, $p_2$ and $q$, $q_1$, $q_2$ such that $\nicefrac{1}{p_1} + \nicefrac{1}{p_2} = \nicefrac{1}{p}$ and $\nicefrac{1}{q_1} + \nicefrac{1}{q_2} = \nicefrac{1}{q}$, it holds that
\begin{equation}\label{eq:Holder anisotropic}
    \| f g \|_{\rLz{q} \rLxy{p}} \le \| f \|_{\rLz{q_1} \rLxy{p_1}} \| g \|_{\rLz{q_2} \rLxy{p_2}}.
\end{equation}

Moreover, we will also use embedding relations separately in $z$ and $(x,y)$, so we observe that
\begin{equation*}
    \begin{aligned}
        &\rWz{r}{q}\rWxy{s}{p} \hra \rWz{r_1}{q_1}\rWxy{s}{p} &&\for \rWz{r}{q}(a,b) \hra \rWz{r_1}{q_1}(a,b), \andtext\\
        &\rWz{r}{q}\rWxy{s}{p} \hra \rWz{r}{q}\rWxy{s_1}{p_1} &&\for \rWxy{s}{p}(G) \hra \rWxy{s_1}{p_1}(G).
    \end{aligned}
\end{equation*}
In addition, we remark that $\rWsq{r+s}(\Omega) \subset \rWz{r}{q}\rWxy{s}{q}$ is valid provided $p = q$, and we point out that these relations are also true for Bessel potential spaces.

Analogously, for $s,t \ge 0$ and $1 \le p_0,p_1,q_0,q_1 \le \infty$, we set $\rBrpqz{r}{q_0}{p_0}\rBrpqxy{s}{q_1}{p_1} := \rBrpq{r}{q_0}{p_0}((a,b);\rBrpq{s}{q_1}{p_1}(G))$ and endow these spaces with the corresponding norms.
For $s,t \ge 0$, the above identity in the case of Sobolev and Bessel potential spaces remains valid, i.\ e., it holds that $\rBrpq{r+s}{q}{p}(\Omega) \subset \rBrpqz{r}{q}{p}\rBrpqxy{s}{q}{p}$.

The next result provides analogous estimates as in \cite[Lemma 5.1]{HK:16} or \cite[Lemma 6.1]{GGHHK:20} for our setting.

\begin{prop}\label{prop:estimates bilinearity PE}
	Let $p,q \in (1,\infty)$ such that $\nicefrac{1}{p} + \nicefrac{1}{q} < \nicefrac{1}{2}$, and consider the bilinear map $\rF$ as defined in~\eqref{eq:bilinearity primitive eq}. 
	Then for $\rF \colon \rBrpq{2-\nicefrac{2}{p}}{q}{p}(\Omega)^2 \cap \rLqOmegasigmabar \times \rBrpq{2-\nicefrac{2}{p}}{q}{p}(\Omega)^2 \cap \rLqOmegasigmabar \to \rLqOmegasigmabar$, we have the following:
	\begin{enumerate}[(a)]
		\item 
		For $v \in \rBrpq{2-\nicefrac{2}{p}}{q}{p}(\Omega)^2 \cap \rLqOmegasigmabar$,
		there exists a constant $C > 0$, depending only on $\Omega$ and $q$, such that
		\begin{equation*}
			\| \rF(v,v) \|_{\rLqOmegasigmabar} \leq C \| v \|_{\rBrpq{2-\nicefrac{2}{p}}{q}{p}(\Omega) \cap \rLqOmegasigmabar}^2.
		\end{equation*}
		\item
		There exists a constant $C > 0$ such that for $v, v' \in \rBrpq{2-\nicefrac{2}{p}}{q}{p}(\Omega)^2 \cap \rLqOmegasigmabar$, we have
		\begin{equation*}
			\| \rF(v) - \rF(v') \|_{\rLqOmegasigmabar} \leq 
			C \Bigl( \| v \|_{\rBrpq{2-\nicefrac{2}{p}}{q}{p}(\Omega) \cap \rLqOmegasigmabar}+\| v'\|_{\rBrpq{2-\nicefrac{2}{p}}{q}{p}(\Omega) \cap \rLqOmegasigmabar}\Bigr) \cdot \| v - v' \|_{\rBrpq{2-\nicefrac{2}{p}}{q}{p}(\Omega) \cap \rLqOmegasigmabar} . 
		\end{equation*}
	\end{enumerate} 
\end{prop}

\begin{proof}
   Since $\bilin{v}{v'}$ is bilinear, we observe that assertion (b) can be shown in a similar way as (a).
    More precisely, we obtain
    \begin{equation*}
        \begin{aligned}
            \| \rF(v) - \rF(v') \|_{\rLqOmegasigmabar} &\le \| \rF(v,v-v') \|_{\rLqOmegasigmabar} + \| \rF(v-v',v') \|_{\rLqOmegasigmabar}\\
            &\le C \Bigl( \| v \|_{\rBrpq{2-\nicefrac{2}{p}}{q}{p}(\Omega) \cap \rLqOmegasigmabar}+\| v'\|_{\rBrpq{2-\nicefrac{2}{p}}{q}{p}(\Omega) \cap \rLqOmegasigmabar}\Bigr) \cdot \| v - v' \|_{\rBrpq{2-\nicefrac{2}{p}}{q}{p}(\Omega) \cap \rLqOmegasigmabar},
        \end{aligned}
    \end{equation*}
    so it suffices to prove (a).
    As $\sP \colon \rLq(\Omega)^2 \to \rLqOmegasigmabar$ is bounded and $\rBrpq{2-\nicefrac{2}{p}}{q}{p}(\Omega)^2 \cap \rLqOmegasigmabar \subset \rBrpq{2-\nicefrac{2}{p}}{q}{p}(\Omega)^2$ holds true, it is sufficient to bound the $\rLq(\Omega)$-norms of $\bilinH{v}{v}$ and $\bilinz{v}{v}$ separately by $C \| v \|_{\rBrpq{2-\nicefrac{2}{p}}{q}{p}(\Omega)}^2$ for some $C > 0$.
    Next, we deduce from $\nicefrac{1}{p} + \nicefrac{1}{q} < \nicefrac{1}{2}$ the existence of $\eta > 0$ such that $2 - \nicefrac{2}{p} - \eta - \nicefrac{3}{q} \ge - \nicefrac{3}{3q}$ and $2 - \nicefrac{2}{p} - \eta - \nicefrac{3}{q} \ge 1 - \nicefrac{2}{q}$.
    It then follows by \cite[Theorem~4.6.1]{Tri:78} that
    \begin{equation}\label{eq:embeddings bilin}
        \rBrpq{2-\nicefrac{2}{p}}{q}{p}(\Omega) \hra \rBrpq{2-\nicefrac{2}{p}-\eta}{q}{3q}(\Omega) \hra \rL^{3q}(\Omega) \andtext \rBrpq{2-\nicefrac{2}{p}}{q}{p}(\Omega) \hra \rBrpq{2-\nicefrac{2}{p} - \eta}{q}{\nicefrac{3q}{2}}(\Omega) \hra \rH^{1,\nicefrac{3q}{2}}(\Omega) \hra \rW^{1,\nicefrac{3q}{2}}(\Omega),
    \end{equation}
    where we used that $\rH^{1,\nicefrac{3q}{2}}(\Omega) \hra \rW^{1,\nicefrac{3q}{2}}(\Omega)$ by virtue of $q \ge 2$, see e.\ g.\ \cite[Chapter~2]{Tri:78}.
    Combining the embeddings from \eqref{eq:embeddings bilin} with H\"older's inequality, we obtain
    \begin{equation*}
        \| \bilinH{v}{v} \|_{\rLq(\Omega)} \le C \| v \|_{\rL^{3q}(\Omega)} \| v \|_{\rW^{1,\nicefrac{3q}{2}}(\Omega)} \le C \| v \|_{\rBrpq{2-\nicefrac{2}{p}}{q}{p}(\Omega)}^2,
    \end{equation*}
    so the proof for the first addend of the bilinearity is complete.
    For the second one, we use \eqref{eq:Holder anisotropic} for
    \begin{equation*}
        \| w \dz{v} \|_{\rLq(\Omega)} \le \| w \|_{\rLz{\infty} \rLxy{2q}} \| \dz{v} \|_{\rLz{q} \rLxy{2q}},
    \end{equation*}
    and we now find estimates for $\| w \|_{\rLz{\infty} \rLxy{2q}}$ and $\| \dz{v} \|_{\rLz{q} \rLxy{2q}}$ separately.
    Exploiting $\rWsq{1}(a,b) \hra \rL^{\infty}(a,b)$, Poincar\'e's inequality, the embedding
    \begin{equation*}
        \rBrpq{2-\nicefrac{2}{p} - \eta}{q}{p}(G) \hra \rH^{1,2q}(G) \hra \rW^{1,2q}(G), \text{ $\eta > 0$ small},
    \end{equation*}
    as well as the above observations concerning embeddings of anisotropic Sobolev spaces, Bessel potential spaces and Besov spaces, and using the assumption that $\divH v + \dz{w} = 0$, we infer that
    \begin{equation*}
        \begin{aligned}
            \| w \|_{\rLz{\infty} \rLxy{2q}} &\le C \| w \|_{\rWz{1}{q} \rLxy{2q}} \le C \| \dz{w} \|_{\rLz{q} \rLxy{2q}} \le C \| \divH v \|_{\rLz{q} \rLxy{2q}}\\
            &\le C \| v \|_{\rLz{q} \rWxy{1}{2q}} \le C \| v \|_{\rBrpqz{\eta}{q}{p} \rBrpqxy{2-\nicefrac{2}{p}-\eta}{q}{p}} \le C \| v \|_{\rBrpq{2-\nicefrac{2}{p}}{q}{p}(\Omega)}.
        \end{aligned}
    \end{equation*}
    
    The condition on $p$ and $q$ yields $1 - \nicefrac{2}{p} - 2 \eta - \nicefrac{2}{q} \ge - \nicefrac{2}{2q}$ for some $\eta > 0$, so \cite[Theorem~4.6.1]{Tri:78} implies
    \begin{equation*}
        \rBrpq{1 - \nicefrac{2}{p} - \eta}{q}{p}(G) \hra \rBrpq{1 - \nicefrac{2}{p} - 2 \eta}{q}{2q}(G) \hra \rL^{2q}(G).
    \end{equation*}
    
    Using the above remarks on the relations of anisotropic Sobolev spaces again and invoking the last embedding, we obtain
    \begin{equation*}
        \begin{aligned}
            \| \dz{v} \|_{\rLz{q} \rLxy{2q}} \le C \| v \|_{\rWz{1}{q} \rLxy{2q}} \le C \| v \|_{\rBrpqz{1+\eta}{q}{p} \rBrpqxy{1 - \nicefrac{2}{p} - \eta}{q}{p}} \le C \| v \|_{\rBrpq{2 - \nicefrac{2}{p}}{q}{p}(\Omega)}.
        \end{aligned}
    \end{equation*}
    The claim follows then by concatenating the previous estimates.
\end{proof}

We get the following corollary by recalling the concrete shape of $\sYzero$ and $\sYgamma$ from \eqref{eq:Y spaces}.

\begin{cor}\label{cor:estimates bilinearity PE atm ocn}
	Let $p,q \in (1,\infty)$ such that $\nicefrac{1}{p} + \nicefrac{1}{q} < \nicefrac{1}{2}$, $v_{\atm,1}, v_{\atm,2} \in \sYgammaa$ and $v_{\ocn,1}, v_{\ocn,2} \in \sYgammao$. Then the following holds.
	\begin{enumerate}[(a)]
		\item 
		There exists a constant $C > 0$ such that
		\begin{equation*}
		    \begin{aligned}
		        \| \Fatm(v_{\atm,1}) - \Fatm(v_{\atm,2}) \|_{\sYzeroa} &\leq 
			    C \bigl( \| v_{\atm,1} \|_{\sYgammaa}+\| v_{\atm,2}\|_{\sYgammaa} \bigr) \| v_{\atm,1} - v_{\atm,2} \|_{\sYgammaa}.
		    \end{aligned}
		\end{equation*}
		\item 
		There exists a constant $C > 0$ such that
		\begin{equation*}
		    \begin{aligned}
		        \| \Focn(v_{\ocn,1}) - \Focn(v_{\ocn,2}) \|_{\sYzeroo} &\leq 
			    C \bigl( \| v_{\ocn,1} \|_{\sYgammao}+\| v_{\ocn,2}\|_{\sYgammao} \bigr) \| v_{\ocn,1} - v_{\ocn,2} \|_{\sYgammao} .
		    \end{aligned}
		\end{equation*}
	\end{enumerate} 
\end{cor}

\section{The Stationary Hydrostatic Stokes Problem}
\label{sec:stationary problem}

This section is dedicated to the analysis of the stationary hydrostatic Stokes problem with inhomogeneous boundary conditions. 
The stationary hydrostatic Stokes problem corresponds to the coupling condition imposed in \eqref{eq:coupling} as well as the boundary condition for $\vocn$ on the lower boundary as made precise in \eqref{eq:boundary conditions v}, and it is given by
\begin{equation}
	\left\{
	\begin{aligned} 
		\Amocn \vocn &= 0, &&\onOmegaocn, \\
		\trGao \vocn &= \varphi, &&\on \Gao, \\
		\trGab \vocn &= 0, &&\on \Gab, 
	\end{aligned} 
	\right.
	\label{eq:inhomogeneous hydrostatic Stokes operator}
\end{equation}
on $\rLqocn$ with $\varphi \in \rLq(G)^2$. 
We denote the solution operator of \eqref{eq:inhomogeneous hydrostatic Stokes operator} by $\Llambda{0} \colon \rLq(\Gao)^2 \to \rLqocn$ and call it the {\em hydrostatic Dirichlet operator}. 
In the main part of this section, we show that $\Llambda{0}$ is well-defined and bounded. 
This operator and its properties are essential for the decoupling approach in \autoref{sec:decoupling}.

\subsection{The hydrostatic Dirichlet operator}
\label{ssec:Dirichlet operator}

\

We start by solving \eqref{eq:inhomogeneous hydrostatic Stokes operator} for smooth functions on the boundary. 
We need the following auxiliary result which guarantees the existence of an extension with vertical average $0$. 
To this end, we construct a smooth solution to the associated extension problem
\begin{equation}
	\left\{
	\begin{aligned}
		\bar{g} &= 0, &&\onOmegaocn, \\
		\trGao g &= \varphi, &&\on \Gao, \\
		\trGab g &= 0, &&\on \Gab.
	\end{aligned}
	\right.
	\label{eq:extension}
\end{equation}

The auxiliary result reads as follows.

\begin{lem}
	Let $\varphi \in \rC_\per^\infty(G)^2$. Then there exists a smooth solution $g \in \rC_\per^\infty(\Omegaocn)^2$ of \eqref{eq:extension}.
\end{lem} 
\begin{proof}
	Note that $\Omegaocn = G \times (-\hocn,0)$. 
	The approach for $g$ is a splitting into the horizontal and the vertical part $g(\xH,z) = r(z) \cdot \varphi(\xH)$. 
	Therefore, the problem \eqref{eq:extension} reduces to the construction of a smooth solution $r$ to the one-dimensional extension problem with mean value $0$ given by
	\begin{equation}
		\left\{
		\begin{aligned}
			\bar{r} &= 0, &&\on{(-\hocn,0)}, \\
			r(0) &= 1, && \\
			r(-\hocn) &= 0. &&
		\end{aligned}
		\right.
		\label{eq:extension 1D}
	\end{equation}
	
	A solution of \eqref{eq:extension 1D} is given by $r(z) = \nicefrac{3}{\hocn^2} z^2 + \nicefrac{4}{\hocn} z + 1$, and it is smooth, i.\ e., $g = r \cdot \varphi \in \rC_\per^\infty(\Omegaocn)^2$.
\end{proof}

\begin{rem}
\begin{enumerate}[(a)]
    \item In general, the solutions of \eqref{eq:extension} are not unique.
    \item Solving \eqref{eq:extension} guarantees in particular that $\divH \bar{g} = 0$ for all $\varphi \in \rC_\per^\infty(G)^2$.
    Functions of the shape $g(\xH,z) = r(z) \cdot \varphi(\xH)$ satisfy $\divH \bar{g} = 0$ for all $\varphi \in \rC_\per^\infty(G)^2$ if and only if $\bar{r}$ vanishes.
\end{enumerate}
\end{rem}

The regularity and average $0$ of the solution $g$ of \eqref{eq:extension} ensure $g \in D(\Amocn)$, so $v_g := \vocn - g \in D(\Amocn)$ satisfies 
\begin{equation}
	\left\{
	\begin{aligned} 
		\Amocn v_g &= f, &&\onOmegaocn, \\
		\trGao v_g &= 0, &&\on \Gao, \\
		\trGab v_g &= 0, &&\on \Gab, 
	\end{aligned} 
	\right.
	\label{eq:vg}
\end{equation}
with $f := -\Amocn g \in \rLqocn$.
Consequently, \eqref{eq:inhomogeneous hydrostatic Stokes operator} admits a unique solution if and only if \eqref{eq:vg} does so.
Note that \eqref{eq:vg} is equivalent to $\An v_g = f$. 
The next result follows from $0 \in \rho(\An)$ in \autoref{cor:Hinfty Stokes atm ocn}(b).

\begin{lem}\label{lem:L_0 densely defined}
	For $\varphi \in \rC_\per^\infty(G)^2$, there exists a unique solution $u \in D(\Amocn)$ of \eqref{eq:inhomogeneous hydrostatic Stokes operator}.
\end{lem}

As a consequence of \autoref{lem:L_0 densely defined}, $\Llambda{0}$ exists as a densely defined operator with domain $\rC_\per^\infty(G)^2$.
This allows us to define a unique adjoint $\Llambda{0}' \colon D(\Llambda{0}') \subset \rLqocn' \to (\rLq(G)^2)'$.
Observe that $\rLqocn' \cong \rLpocn$
and $(\rLq(G)^2)' \cong \rL^{q'}(G)^2$
with $\nicefrac{1}{q}+\nicefrac{1}{q'} = 1$.

We denote by $\ddnr{r} \colon D(\ddnr{r}) \subset \rLrocn{r} \to \rL^r(G)^2$ the (distributional) derivative in the direction of $\nGao$ with maximal domain $D(\ddnr{r}) := \{ f \in \rLrocn{r} \colon \ddnr{r} f \in \rL^r(G)^2 \}$ for $r \in (1,\infty)$. 
Moreover, we denote by $\Anr{r}$ the operator with homogeneous boundary conditions given in \eqref{eq:boundary conditions A0} on $\rLrocn{r}$ for $r \in (1,\infty)$. 
We remark that  $(\Anq)' = \Anp$ with $\nicefrac{1}{q}+\nicefrac{1}{q'} = 1$. 
The adjoint $\Llambda{0}'$ is related to the hydrostatic Stokes operator with homogeneous boundary conditions and the normal derivative in the following way.

\begin{lem}\label{lem:L_0 adjoint}
	Let $q, q' \in (1,\infty)$ such that $\nicefrac{1}{q}+\nicefrac{1}{q'} = 1$.
	The adjoint $\Llambda{0}'$ of $\Llambda{0}$ satisfies
	\begin{equation}
		\Llambda{0}' = \ddnp R(0,\Anq)' = \ddnp R(0,\Anp) .
		\label{eq:adjoint L0}
	\end{equation}
\end{lem}
\begin{proof}	
    Let $\varphi \in \rC_\per^\infty(\Gao)^2$ and $k \in \rLpocn $ with $\nicefrac{1}{q}+\nicefrac{1}{q'} = 1$, so $(\Anq)' = \Anp$.
	We set $f := \Llambda{0} \varphi$ and $g := (\Anp)^{-1} k$.
	Then we have $f \in D(\Amocn) \subset \rLqocn$, $g \in D(\Anp) \subset \rLpocn$ with $\nicefrac{1}{q}+\nicefrac{1}{q'} = 1$, and it follows therefrom, also invoking the periodic boundary conditions on the lateral boundary, that
	\begin{equation*}
		\scalarprodLtwo{\nablaH \pi}{g}{\Omegaocn}
		= \int_G \int_{-\hocn}^0 \nablaH \pi \cdot g \d z \d \xH
		= \hocn \intG{\nablaH \pi \cdot \bar{g}}
		= -\hocn \intG{\pi \cdot \divH \bar{g}} = 0,
	\end{equation*}
    and likewise $\scalarprodLtwo{f}{\nablaH \pi}{\Omegaocn} = 0$.
	Green's second identity and the horizontal periodicity then imply that
	\begin{equation*}
	    \begin{aligned}
        &\scalarprodLtwo{\Delta f + \nablaH \pi}{g}{\Omegaocn}
        -\scalarprodLtwo{f}{\Delta g + \nablaH \pi}{\Omegaocn}\\
		=\,& \scalarprodLtwo{\Delta f}{g}{\Omegaocn}
        - \scalarprodLtwo{f}{\Delta g}{\Omegaocn}\\
		=\, &\intG{\dz f(\xH,0) g(\xH,0)}
		+ \intG{\dz f(\xH,-\hocn) g(\xH,-\hocn)} \\
		&-\intG{f(\xH,0) \dz g(\xH,0)}
		- \intG{f(\xH,-\hocn) \dz g(\xH,-\hocn)}.	        
	    \end{aligned}
	\end{equation*}

	Since $g \in D(\Anp)$, we have $g(\xH,0) = g(\xH,-\hocn) = 0$, and it also holds that $f(\xH,-\hocn) = 0$ as well as $f(\xH,0) = \varphi$.
    This yields that
    \begin{equation*}
        \scalarprodLtwo{f}{\Delta g + \nablaH \pi}{\Omegaocn} = \scalarprodLtwo{\Delta f + \nablaH \pi}{g}{\Omegaocn}
        + \intG{\varphi(\xH) \partial_z g(\xH,0)}.
    \end{equation*}

    Making use of the latter identity in conjunction with the definition of the hydrostatic Stokes operator as in \autoref{ssec:hydrostatic Helmholtz and Stokes} as well as $\Amocn f = 0$ on $\Omegaocn$ by construction, we derive that
    \begin{equation*}
        \begin{aligned}
            \scalarprodLtwo{\Llambda{0} \varphi}{k}{\Omegaocn} &= \scalarprodLtwo{f}{\Anp g}{\Omegaocn} = \scalarprodLtwo{f}{\Delta g + \nablaH \pi}{\Omegaocn}\\
            &= \scalarprodLtwo{\Delta f + \nablaH \pi}{g}{\Omegaocn} + \intG{\varphi(\xH) \partial_z g(\xH,0)}\\
            &= \scalarprodLtwo{\varphi}{(\dz ((\Anp)^{-1}k))|_{\Gamma_u}}{\Omegaocn},
        \end{aligned}
    \end{equation*}
	so the claim follows. 
\end{proof}

The next lemma guarantees that the right-hand side of \eqref{eq:adjoint L0} can be extended to a bounded operator.

\begin{lem}\label{lem:ddn A0 bdd}
	The normal derivative $\ddnr{r}$ is relatively $(-\Anr{r})^\delta$-bounded with
	$\delta > \frac{1}{2}+\frac{1}{2r}$.
\end{lem}

\begin{proof}
    From \cite[Theorem~VIII.1.3.1]{Ama:19}, establishing the result for the half-space, and the observation that horizontally periodic functions can be extended periodically onto a layer and then cut off, we deduce for  $r > 1 + \nicefrac{1}{q}$ the embedding $\rHsqper{r}(\Omegaocn)^2 \cap \rLqocn \hra  D(\ddnq)$.
    The claim is then a consequence of \autoref{cor:Hinfty Stokes atm ocn}, implying that $-\An$ has bounded imaginary powers, see also the relation \eqref{eq:rel Hinfty BIP RS}.
\end{proof}

Combining \autoref{lem:L_0 adjoint} and \autoref{lem:ddn A0 bdd}, we conclude that $\Llambda{0}'$ can be extended to a bounded operator from $\rL^{q'}_{\sigmabar}(\Omegaocn)$ to $\rL^{q'}(G)^2$. 
We then get the following result.

\begin{prop}\label{prop:L_0 existence}
	The operator $\Llambda{0}$ can be extended to a bounded operator from $\rLq(G)^2$ to $\rLqocn$ which will be denoted again by $\Llambda{0}$ for the sake of simplicity. 
	In particular, the unique solution $v$ of \eqref{eq:inhomogeneous hydrostatic Stokes operator} satisfies the a-priori estimate $\| \vocn \|_{\rLqocn} \leq C \| \varphi \|_{\rLq(G)}$.
\end{prop}

As we see in the next corollary, more regularity of the boundary data implies more regularity of the solution. 

\begin{cor}\label{cor:L_0 regularity}
	Let $q \in (1,\infty)$ and $r < \nicefrac{1}{q}$.
	If $\varphi \in \rLq(G)^2$, then the unique solution $\vocn$ of \eqref{eq:inhomogeneous hydrostatic Stokes operator} satisfies $\vocn \in \rHsqper{r}(\Omegaocn)^2 \cap \rLqocn$.
	In addition, for $s > 0$ and $\varphi \in \rHsqper{s}(G)^2$, the unique solution $\vocn$ of \eqref{eq:inhomogeneous hydrostatic Stokes operator} satisfies $\vocn \in \rHsqper{s+r}(\Omegaocn)^2 \cap \rLqocn$. 
    Furthermore, for $s > 0$ and $\varphi \in \rBsqpper{s}(G)^2$, the unique solution $\vocn$ of \eqref{eq:inhomogeneous hydrostatic Stokes operator} satisfies $\vocn \in \rBsqpper{s+r}(\Omegaocn)^2 \cap \rLqocn$.
    
	The operator $\Llambda{0}$ is also bounded from $\rHsqper{s}(G)^2$ to $\rHsqper{s}(\Omegaocn)^2$ and from $\rBsqpper{s}(G)^2$ to $\rBsqpper{s}(\Omegaocn)^2$, i.\ e.,  the solution $\vocn$ of \eqref{eq:inhomogeneous hydrostatic Stokes operator} satisfies the a-priori bounds $\| \vocn \|_{\rHsqper{s}(\Omegaocn)} \leq C \| \varphi \|_{\rHsqper{s}(G)}$ as well as $\| \vocn \|_{\rBsqpper{s}(\Omegaocn)^2} \leq C \| \varphi \|_{\rBsqpper{s}(G)}$.
\end{cor}
\begin{proof}
	First, we consider $s \in (0,2-\nicefrac{1}{q}]$ and $\varphi \in \rHsqper{s}(G)^2$.
	From \autoref{prop:L_0 existence} it follows that there exists a unique solution $v \in \rLqocn$ to \eqref{eq:inhomogeneous hydrostatic Stokes operator} such that $\Pocn \Delta \vocn = \Amocn \vocn = 0 \in \rLqocn$.
	Surjectivity of the hydrostatic Helmholtz projection $\Pocn \colon \rLq(\Omegaocn)^2 \to \rLqocn$ yields $\Delta \vocn \in \rLq(\Omegaocn)^2$. 
	In particular, $\vocn$ solves the inhomogeneous Laplace problem given by
	\begin{equation}
		\left\{
		\begin{aligned}
			\Delta \vocn &= \nablaH \pi, &&\onOmegaocn, \\
			\trGao \vocn &= \varphi, &&\on{\Gao},\\
			\trGab \vocn &= 0, &&\on{\Gab} .
		\end{aligned}
		\right.
		\label{eq:inhomogen Laplace}
	\end{equation}
	
	It follows from $\Delta \vocn \in \rLq(\Omegaocn)^2$ that the right-hand side satisfies $\nablaH \pi \in \rLq(\Omegaocn)^2$.
	Standard regularity theory of the Laplacian implies $\vocn \in \rHsqper{s+r}(\Omegaocn)^2$ with $r < \nicefrac{1}{q}$.
	
	Assume now $s \geq 2 - \nicefrac{1}{q}$ and $\varphi \in \rHsqper{s}(G)^2$. 
	By the above argument, we have $\vocn \in \rHsqper{2}(\Omegaocn)^2 \cap \rLqocn$.
	Using the hydrostatic solenoidal condition, we obtain by a direct calculation similar to \cite[(4.3)]{GGHHK:17} that the pressure term is given by
	\begin{equation}
		\nablaH \pi = \frac{1}{\hocn} \nablaH \DeltaH^{-1} \divH \ddnGao v .
		\label{eq:pressure 2}
	\end{equation}
	
	We conclude $\nablaH \pi \in \rHsqper{1-\nicefrac{1}{q}}(G)^2 \hra \rHsqper{1-\nicefrac{1}{q}}(\Omegaocn)^2$.
	The general case $s \in \R_+$ follows by a bootstrap argument.  
	
	Note that for $s = 0$, the a-priori bound is shown in \autoref{prop:L_0 existence}.
	We first prove the a-priori bound for $s \in \N$ with $s \geq 2$. The proof uses induction over $s \in \{r \in \N \colon r \geq 2\}$.
	Using classical Calder\'on-Zygmund theory, we obtain from \eqref{eq:inhomogen Laplace} that
	\begin{equation}
		\| \vocn \|_{\rHsqper{2}(\Omegaocn)}
		\leq C ( \| \varphi \|_{\rHsqper{2}(G)} + \| \nablaH\pi \|_{\rLq(G)} ).
		\label{eq:calderon-zygmund}
	\end{equation}
	
	From \eqref{eq:pressure 2}, the boundedness of the normal derivative by \autoref{lem:ddn A0 bdd}, the interpolation inequality, Young's inequality and the a-priori estimate on $\rLq$ as in \autoref{prop:L_0 existence}, it follows that for all $r > 1 + \nicefrac{1}{q}$, $\delta = \frac{r}{2} \in (0,1)$ and $\eta > 0$,
	\begin{equation*}
		\begin{aligned}
			\| \nablaH \pi \|_{\rLq(G)}
			&\leq C \| \ddnGao \vocn \|_{\rLq(G)} \\
			&\leq C \| \vocn \|_{\rHsqper{r}(\Omegaocn)} \\
			&\leq C \| \vocn \|_{\rHsqper{2}(\Omegaocn)}^{\delta}
			\cdot \| \vocn \|_{\rLq(\Omegaocn)}^{1-\delta} \\
			&\leq C \delta \eta^{\nicefrac{1}{\delta}} \| \vocn \|_{\rHsqper{2}(\Omegaocn)}
			+ C (1-\delta) \eta^{-\nicefrac{1}{(1-\delta)}} \| \varphi \|_{\rLq(G)}. 
		\end{aligned} 
	\end{equation*} 
	
	Plugging this estimate into \eqref{eq:calderon-zygmund}, we obtain the desired $\rHsq{2}$-a-priori bound by an absorption argument.
	
	Assume that the a-priori bound still holds for $s-1$. 
	Replacing the Calder\'on-Zygmund a-priori estimate~\eqref{eq:calderon-zygmund} by its higher-order analogue and using the higher-order bounds for the normal derivative and the fact that the pseudo-differential operator $\nablaH\DeltaH^{-1}\divH$ of order $0$ is bounded on $\rHsqper{s}(G)^2$, we deduce the induction step from the same argument as the induction base.
 
	The general case $s \in \R_+$ follows by interpolation, and the additional regularity of the solution in the Besov space scale follows from the embeddings $\rBsqpper{s}(G) \hra \rHsqper{s-\delta}(G)$ and $\rHsqper{s+\nicefrac{1}{q}}(\Omegaocn) \hra \rBsqpper{s+\nicefrac{1}{q}-\delta}(\Omegaocn)$ for $\delta > 0$, and the corresponding results for Bessel potential spaces.
	Finally, the boundedness of $\Llambda{0}$ between Besov spaces follows from the boundedness of $\Llambda{0}$ between $\rHsqper{k}(G)^2$ and $\rHsqper{k}(\Omegaocn)^2$ for $k \in \N_0$ by real 
interpolation.
\end{proof}

Combing $0 \in \rho(\An)$ from \autoref{cor:Hinfty Stokes atm ocn} with \autoref{prop:L_0 existence} and using the linearity of the hydrostatic Stokes operator, we obtain that the general inhomogeneous hydrostatic Stokes problem given by
\begin{equation*}
	\left\{
	\begin{aligned} 
		\Amocn \vocn &= f, &&\onOmegaocn, \\
		\trGao \vocn &= \varphi, &&\on \Gao, \\
		\trGab \vocn &= 0, &&\on \Gab, 
	\end{aligned} 
	\right.
\end{equation*}
for $f \in \rLqocn$ and $\varphi \in \rLq(G)^2$, admits a unique solution $\vocn \in D(\Am)$, with $\vocn = -R(0,\An)f + \Llambda{0} \varphi$. 

\begin{rem}\label{rem:andere RBs L0}
	Analogous results about stationary hydrostatic Stokes problems for pure Dirichlet or pure Neumann boundary conditions can be shown by the same technique. 
    First of all, note that it is sufficient to consider only inhomogeneous boundary conditions, since the general case follows from linearity of the problem.
	The case of pure (inhomogeneous) Neumann boundary conditions is treated in \cite[Section~3]{BHHS:22}.
\end{rem}

\subsection{The hydrostatic Dirichlet-to-Neumann operator}
\label{ssec:Dirichlet-to-Neumann operator}

\

The hydrostatic Dirichlet-to-Neumann operator on $\rLq(G)^2$ is given by the composition of the normal derivative and the Dirichlet operator, i.\ e.,
\begin{equation*}
	\Nlambda{0} \varphi := \ddnq \Llambda{0} \varphi, 
	\quad D(\Nlambda{0}) := \left\{ \varphi \in \rLq(G)^2 \colon \Llambda{0} \varphi \in D(\ddnq) \right\} .
\end{equation*}

Using the regularity theory for the inhomogeneous stationary Stokes problem, we obtain the following result about the domain of the Dirichlet-to-Neumann operator.

\begin{prop}\label{prop: D(N)}
	The domain of the Dirichlet-to-Neumann operator contains $\rHsqper{s}(G)^2$ for all $s > 1$. 
\end{prop}
\begin{proof}
	As in the proof of \autoref{lem:ddn A0 bdd}, we infer that $\rHsqper{r}(\Omegaocn)^2 \cap \rLqocn \subset D(\ddnq)$ for $r > 1+\frac{1}{q}$. 
	\autoref{cor:L_0 regularity} yields $\Llambda{0} \rHsqper{s}(G)^2 \subset \rHsqper{r}(\Omegaocn)^2 \cap \rLqocn \subset D(\ddnq)$ for $s > 1$, i.\ e.,  $\rHsqper{s}(G)^2 \subset D(\Nlambda{0})$.
\end{proof}

It follows by \autoref{prop:MR and Hinfty AH} that $-\AH + \omega$ has bounded imaginary powers, see also \eqref{eq:rel Hinfty BIP RS}.
We deduce therefrom that $D((-\AH)^{\delta}) \hra \rHsqper{2\delta}(G)^2$ and then conclude the following result.

\begin{cor}\label{cor:N AH bdd} 
	The Dirichlet-to-Neumann operator $\Nlambda{0}$ is relatively $(-\AH)^\delta$-bounded for $\delta > \frac{1}{2}$.
\end{cor}

\section{Decoupling}
\label{sec:decoupling}

In this section, we combine the results from \autoref{sec:Hibler}, \autoref{sec:pe} and \autoref{sec:stationary problem} to conclude maximal regularity of the operator matrix given in \eqref{eq:operator matrix} as well as Lipschitz estimates of the operator matrix given in \eqref{eq:operator matrix} and the nonlinearity \eqref{eq:nonlinearity}. 

\subsection{Maximal regularity and bounded \protect{$\Hinfty$}-calculus of the linearized operator matrix}
\label{sec:MR operator matrix}

\

Fix $(v_{\ice,0},h_0,a_0) \in \Vice$ and consider the operator matrix $\sA(v_{\ice,0},h_0,a_0)$. 
Throughout this section, we omit the fixed variable $(v_{\ice,0}, h_0,a_0)$ and denote the operator matrix by $\sA := \sA(v_{\ice,0}, h_0,a_0)$.
However, we impose the required regularity conditions on $(v_{\ice,0}, h_0,a_0)$ when necessary.
Note that $\sA$ does not have a diagonal domain. 
To deal with this issue, we follow the ideas from \cite{BtE:24}.
We introduce the operator matrix $\tsA \colon D(\tsA) \subset \sX_0 \to \sX_0$ given by
\begin{equation*}
	\begin{aligned}
		\tsA 
		&:= 
		\begin{pmatrix}
			\Aatm & 0 & 0 & 0 & 0 \\
			0 & \An + \Llambda{0} \Coi \ddnGao & - \Llambda{0} (\AH - \Coi \Nlambda{0}) & 0 & 0 \\
			0 & -\Coi  \ddnGao & \AH - \Coi \Nlambda{0} & \Bh & \Ba \\
			0 & 0 & 0 & \Deltah & 0 \\
			0 & 0 & 0 & 0 & \Deltaa 
		\end{pmatrix}, \\
		D(\tsA) &:= 
		D(\Aatm) \times D(\An) \times D(\AH) \times D(\DeltaH) \times D(\DeltaH) .
	\end{aligned}
\end{equation*}

The next lemma investigates the relationship between the operator matrices $\sA$ and $\tsA$. 

\begin{lem}\label{lem:aet}
	The operator matrices $\sA$ and $\tsA$ are isomorphic. 
\end{lem}
\begin{proof}
	The similarity transform is given by
	\begin{equation*}
		\sS := 
		\begin{pmatrix}
			\Id & 0 & 0 & 0 & 0 \\
			0 & \Id & -\Llambda{0} & 0 & 0 \\
			0 & 0 & \Id & 0 & 0 \\
			0 & 0 & 0 & \Id & 0 \\
			0 & 0 & 0 & 0 & \Id
		\end{pmatrix}, \quad \text{with inverse} \quad \sS^{-1} := 
		\begin{pmatrix}
			\Id & 0 & 0 & 0 & 0 \\
			0 & \Id & \Llambda{0} & 0 & 0 \\
			0 & 0 & \Id & 0 & 0 \\
			0 & 0 & 0 & \Id & 0 \\
			0 & 0 & 0 & 0 & \Id
		\end{pmatrix}.
	\end{equation*}
	
	From \autoref{prop:L_0 existence} we conclude that $\sS \in \sL(\sX_0)$, and that the inverse $\sS^{-1}$ is also bounded on $\sX_0$. 
	A direct calculation shows that $\tsA = \sS \sA \sS^{-1}$.
    Standard regularity theory implies that $\Tilde{v}_\ocn \in D(\Amocn)$ with homogeneous boundary conditions fulfills $\Tilde{v}_\ocn \in D(A_0^\ocn)$.
    Consequently, since $\Tilde{v}_\ocn \coloneqq \vocn - L_0 \vice$ for $v \in D(A) = \sX_1$ has these properties, we conclude $D(\tsA) = \sS D(\sA)$.
\end{proof}

For $v = (\vatm,\vocn,\vice,h,a) \in \sX_1$, we conclude from \autoref{lem:aet} and its proof that $\vocn = \Tilde{v}_\ocn + L_0 \Tilde{v}_\ice$, where $\Tilde{v}_\ocn \in D(A_0^\ocn)$ as well as $\Tilde{v}_\ice \in \rW_\per^{2,q}(G)^2$. 
Together with \autoref{cor:L_0 regularity}, this results in the relation $\vocn \in \rW_\per^{2,q}(\Omegaocn)^2 \cap \rLsigmabar(\Omegaocn) \eqqcolon D(A^\ocn)$.
Therefore, we have
\begin{equation}\label{eq:rel of sX_1}
    \begin{aligned}
        \sX_1 
        &= \{v \in D(A^\atm) \times D(A^\ocn) \times D(\AH) \times D(\DeltaH) \times D(\DeltaH) :\\
        &\qquad \trGab\vocn = 0, \text{ on } \Gab, \text{ and } \trGao\vocn = \vice, \text{ on } \Gao\}.
    \end{aligned}
\end{equation}

Note that the operator $\tsA$ has a diagonal domain and therefore, it is possible to work component-wise. 
To establish $\Hinfty$-calculus, we decompose the operator matrix $\tsA$ into smaller submatrices. 
More precisely, we use the following representations
\begin{equation}\label{eq:decomposition tsA}
    \tsA = \begin{pmatrix}
		\Aatm & 0 \\
		0 & \J
	\end{pmatrix}, \quad D(\tsA) = D(\Aatm) \times D(\J),
\end{equation}
where the operator matrix $\J$ is split into
\begin{equation}\label{eq:decomposition J}
    \J = \begin{pmatrix}
			\Je & \B' \\
			0 & \Jz
		\end{pmatrix}, \quad D(\J) = D(\Je) \times D(\Jz), \text{ with } \B' = \begin{pmatrix}
			0 & 0 \\
			\Bh & \Ba
		\end{pmatrix} \andtext \Jz = \diag{\Deltah, \Deltaa}.
\end{equation}

Moreover, $\Je$ is given by
\begin{equation*}
    \Je = \begin{pmatrix}
		\An + \Llambda{0} \Coi  \ddnGao & - \Llambda{0}(\AH - \Coi  \Nlambda{0}) \\
		-\Coi  \ddnGao & \AH - \Coi  \Nlambda{0}
	\end{pmatrix}, \quad D(\Je) = D(\An) \times D(\AH).
\end{equation*}

The crucial step is to show that $\Je$ admits a bounded $\Hinfty$-calculus, which requires some preparation.

For convenience of the reader, we recall the notion of diagonal dominance of a block operator matrix from \cite[Assumption~3.1]{AH:23}.
Let $X$, $Y$ be Banach spaces and consider linear operators
\begin{equation*}
    \A \colon D(\A) \subset X \to X, \quad \D \colon D(\D) \subset Y \to Y, \quad \B \colon D(\B) \subset Y \to X \qandq \mathrm{C} \colon D(\mathrm{C}) \subset X \to Y,
\end{equation*}
where $\A$ and $\D$ are assumed to be densely defined and closed.
For $Z = X \times Y$, the block operator matrix
\begin{equation}\label{eq:block op matrix J}
    \mathrm{K} \colon D(\mathrm{K}) = D(\A) \times D(\D) \subset Z \to Z \quad \text{with} \quad \mathrm{K} \binom{x}{y} = \begin{pmatrix}
        \A & \B\\
        \mathrm{C} & \D
    \end{pmatrix} \binom{x}{y} \quad \text{for all} \quad \binom{x}{y} \in D(\mathrm{K})
\end{equation}
is then said to be {\em diagonally dominant} if $D(\D) \subset D(\B)$, $D(\A) \subset D(\mathrm{C})$, and there exist $c_{\A}$, $c_{\D}$, $L \ge 0$ with
\begin{equation*}
    \begin{aligned}
        \| \mathrm{C} x \|_Y &\le c_{\A} \| \A x \|_X + L \| x \|_X \quad \text{for all} \quad x \in D(\A), \quad \text{and}\\
        \| \B y \|_X &\le c_{\D} \| \D y \|_Y + L \| y \|_Y \quad \text{for all} \quad y \in D(\D).
    \end{aligned}
\end{equation*}
We remark that the relative boundedness of $\mathrm{C}$ is especially implied if there exists $\gamma \in (0,1)$ such that $\mathrm{C} \in \sL(D(\A^\gamma),Y)$, see e.\ g.\ \cite[Corollary~5.7]{AH:23}.
For the following result, we also refer to \cite[Corollary~5.7]{AH:23}.

\begin{lem}\label{lem:aux result Hinfty-calculus}
Let $\mathrm{K}$ be a diagonally dominant block operator matrix as in \eqref{eq:block op matrix J}, suppose in addition that $A$ and $D$ are sectorial operators and assume that there exists $\delta \in (0,1)$ such that for some constant $c>0$
\begin{equation}\label{eq:cond aux result}
    \begin{aligned}
        \mathrm{C}(D(\A^{1+\delta})) &\subset D(\D^\delta) \qandq \| \D^\delta \mathrm{C} x \|_Y \le c \| \A^{1+\delta} x \|_X \quad \text{for all} \quad x \in D(\A^{1+\delta}), \quad \text{and}\\
        \B(D(\D^{1+\delta})) &\subset D(\A^\delta) \qandq \| \A^\delta \B y \|_X \le c \| \D^{1+\delta} y \|_Y \quad \text{for all} \quad y \in D(\D^{1+\delta}).
    \end{aligned}
\end{equation}
Then if $\A$ and $\D$ admit a bounded $\Hinfty$-calculus on $X$ and $Y$, respectively, it follows that there is $\omega_0 \in \R$ such that for all $\omega > \omega_0$, the block operator matrix $\mathrm{K} + \omega$ admits a bounded $\Hinfty$-calculus on $Z$.
\end{lem}

We are now in the position to state and prove the $\Hinfty$-calculus of $\J_1$.

\begin{lem}\label{lem:J_1 Hinfty calculus}
	Let $p,q \in (1,\infty)$ such that $\nicefrac{1}{p}+\nicefrac{1}{q} < \frac{1}{2}$ and $(v_{\atm,0},v_{\ocn,0},v_0) \in V$. 
    Then there exists a constant $\omega_0 \in \R$ such that for all $\omega > \omega_0$, $-\Je + \omega$ admits bounded $\Hinfty$-calculus on $\rLqocn \times \rLq(G)^2$.
\end{lem}
\begin{proof}
	We set $\A := \An$, $\B := - \Llambda{0}\AH$, $\mathrm{C} := -\Coi  \ddnGao$ and $\D := \AH$ to mimic the notation from above.
	First, we consider the operator matrix $\tilde{\Je}$ given by
		\begin{equation}
		\tilde{\J}_1
		:=
		\begin{pmatrix}
			\An & - \Llambda{0}\AH \\
			-\Coi  \ddnGao & \AH
		\end{pmatrix}
		= 
		\begin{pmatrix}
			\A & \B \\
			\mathrm{C} & \D 
		\end{pmatrix} 
		\label{eq:tildeJ}
	\end{equation}
	with diagonal domain.
	We now verify the assumptions of \autoref{lem:aux result Hinfty-calculus} for $\tilde{\J}_1$.
    The diagonal dominance of $\Tilde{\J}_1$ follows from \autoref{lem:ddn A0 bdd}, as the latter result also implies that $\mathrm{C} \in \sL((D(-\A)^{\gamma}),\rL^q(G)^2)$ for all $\gamma \in (\nicefrac{1}{2}(1+\nicefrac{1}{q}),1]$.
	It remains to verify \eqref{eq:cond aux result}.
	Let $\delta \in (0,\nicefrac{1}{2q})$. 
    As the bounded $\Hinfty$-calculus of $-\AH + \omega$ and $-\An$ from \autoref{prop:MR and Hinfty AH} and \autoref{cor:Hinfty Stokes atm ocn} especially yields the boundedness of the imaginary powers of these operators, see also the relation \eqref{eq:rel Hinfty BIP RS} below, it follows that
	\begin{equation*}
		\begin{aligned}
			D(\A^{1+\delta}) &\hra \rHsqper{2+2\delta}(\Omegaocn)^2 \cap \rLqocn,
			&&D(\A^{\delta}) = \rHsqper{2\delta}(\Omegaocn)^2 \cap \rLqocn, \\
			D(\D^{1+\delta}) &= \rHsqper{2+2\delta}(G)^2, 
			&&D(\D^{\delta}) = \rHsqper{2\delta}(G)^2 .
		\end{aligned}
	\end{equation*}
	
	As in the proof of \autoref{lem:ddn A0 bdd}, it follows that $\mathrm{C} D(\A^{1+\delta}) \subset \rHsqper{2\delta}(G)^2 = D(\D^{\delta})$. 
	From \autoref{cor:L_0 regularity} and $\B = -\Llambda{0}\D$ it follows that
	$\B D(\D^{1+\delta}) = -\Llambda{0} D(\D^{\delta}) = \Llambda{0} \rHsqper{2\delta}(G)^2 \subset \rHsqper{2\delta}(\Omega)^2 \cap \rLqocn = D(\A^{\delta})$. 
	Since the operators $\B$ and $\mathrm{C}$ are closed, the estimates in \eqref{eq:cond aux result} follow by the closed graph theorem.
	
	Now, \autoref{lem:aux result Hinfty-calculus} implies that $-\tilde{\J}_1+\omega$, given by \eqref{eq:tildeJ}, admits a bounded $\Hinfty$-calculus on $\rLqocn \times \rLq(G)^2$ for some constant $\omega \in \R$. 
	By \autoref{lem:ddn A0 bdd} and \autoref{prop:L_0 existence}, the operator $\Llambda{0} \Coi  \ddnGao$ is relatively $(-\Aocn)^{\gamma}$-bounded for all $\gamma \in (\nicefrac{1}{2}(1+\nicefrac{1}{q}),1]$. 
	In addition, by \autoref{prop:L_0 existence} and \autoref{cor:N AH bdd}, the operators $\Llambda{0} \Coi  \Nlambda{0}$ and $\Coi  \Nlambda{0}$ are relatively
        $(-\AH)^{\theta}$-bounded for all $\theta \in (\nicefrac{1}{2},1)$. 
	Therefore, the claim follows from perturbation theory for the $\Hinfty$-calculus, see for instance \cite[Proposition~13.1]{KW:04}.
\end{proof}

We conclude with the main result  of this section. 

\begin{prop}\label{thm:Hinfty operatormatrix}
	Let $p,q \in (1,\infty)$ such that $\nicefrac{1}{p} + \nicefrac{1}{q} < \nicefrac{1}{2}$ and $(v_{\atm,0},v_{\ocn,0},v_0) \in V$.
	Then there exists a constant $\omega_0 \in \R$ such that for all $\omega > \omega_0$, the operator matrix $-\sA + \omega$, with $\sA$ given in \eqref{eq:operator matrix}, admits a bounded $\Hinfty$-calculus on $\sX_0$.
\end{prop}
\begin{proof}
    In the sequel, the constants $\omega \in \R$ may change from line to line.
	It follows from \cite[Theorem~8.22]{Nau:13} that $- \Jz + \omega$ admits a bounded $\Hinfty$-calculus on $\rLq(G) \times \rLq(G)$ for some $\omega \in \R$. 
	Combining this with \autoref{lem:J_1 Hinfty calculus}, we deduce that $- \diag{\Je,\Jz} + \omega$ with diagonal domain admits a bounded $\Hinfty$-calculus on $\rLqocn \times \rLq(G)^2 \times \rLq(G) \times \rLq(G)$ for some $\omega \in \R$.
	Further, we know from \autoref{cor:Bh Ba relatively AH-bd} that $\B'$ is relatively $(-\AH+\tilde{\omega})^{\frac{1}{2}}$-bounded, and it follows from \cite[Proposition 13.1]{KW:04} and \eqref{eq:decomposition J} that $- \J+\omega$ admits a bounded $\Hinfty$-calculus on $\rLqocn \times \rLq(G)^2 \times \rLq(G) \times \rLq(G)$ for some $\omega \in \R$.
	
	Invoking \autoref{cor:Hinfty Stokes atm ocn}(a), we conclude from \eqref{eq:decomposition tsA} that $-\tsA+\omega$ admits a bounded $\Hinfty$-calculus on $\sX_0$.
 	The claim finally follows from \autoref{lem:aet} together with the fact that the bounded $\Hinfty$-calculus is preserved under similarity transforms, see for instance \cite[Proposition~2.11]{DHP:03}.
\end{proof}

We collect some interesting functional analytic properties.
For the spaces $\sYbeta$ and $\sYtheta$, see \eqref{eq:Y spaces}. 

\begin{cor}\label{cor:functionalanalytic properties opmat}
	Let $p,q \in (1,\infty)$ such that $\nicefrac{1}{p}+\nicefrac{1}{q} < \nicefrac{1}{2}$ and $(v_{\atm,0},v_{\ocn,0},v_0) \in V$.
	Then there exists $\omega_0 \in \R$ such that for $\omega > \omega_0$, the following properties hold true.
	
	\begin{enumerate}[(a)]
		\item The operator $-\sA + \omega$ has an $\sR$-bounded $\Hinfty$-calculus on $\sX_0$ with angle $\Phi_{-\sA + \omega} < \nicefrac{\pi}{2}$.
		
		\item It holds that $-\sA + \omega \in \BIP(\sX_0)$, i.\ e., the operator $-\sA + \omega$ has bounded imaginary powers.
		
		\item The operator $-\sA + \omega$ has the property of maximal $\rLp[0,\infty)$-regularity in $\sX_0$.
	
		\item For $\beta \in (0,1)$ with $\beta \notin \{\nicefrac{1}{2q},\nicefrac{1}{2}+\nicefrac{1}{2q}\}$, the fractional power domains coincide with the complex interpolation spaces, i.\ e., $D((-\sA + \omega)^\beta) \simeq [\sX_0,D(-\sA)]_{\beta} \hra \sYbeta$.		
		
		\item For $\theta \in (0,1)$ with $\theta \notin \{\nicefrac{1}{2q},\nicefrac{1}{2}+\nicefrac{1}{2q}\}$, the real interpolation spaces satisfy $(\sX_0,D(-\sA))_{\theta,p} \hra \sYtheta$.
		
		\item The Riesz transform $\nabla (-\sA + \omega)^{\nicefrac{1}{2}}$ is bounded on $\sX_0$.

        \item The operator matrix $\sA$ has a compact resolvent, and the spectrum $\sigma(\sA)$ of $\sA$ on $\sX_0$ is $q$-independent.
	\end{enumerate} 
\end{cor}

\begin{proof}
  For (a), we observe that the space $\sX_0$ has the property $(\alpha)$. Thus $\sR$-boundedness of the $\Hinfty$-calculus is then equivalent to the $\Hinfty$-calculus and the angles
  coincide, showing (a).

    The assertions~(b) and (c) follow from  \autoref{thm:Hinfty operatormatrix} via the relations
    \begin{equation}\label{eq:rel Hinfty BIP RS}
        \Hinfty(\sX_0) \subset \BIP(\sX_0) \subset \sR \sS(\sX_0) \quad \text{with} \quad \phi_{\sA}^\infty \ge \theta_{\sA} \ge \phi_{\sA}^\sR,
    \end{equation}
    see for example \cite[Section~4.4]{DHP:03}, and the observation that $\sX_0$ is in particular a UMD-space.

	The first part of property~(d) follows from \cite[Theorem 2.5]{DHP:03}.
	Note that for $\tsA$, we obtain from the bounded imaginary powers of Hibler's operator and the hydrostatic Stokes operators in view of \autoref{prop:MR and Hinfty AH} and \autoref{cor:Hinfty Stokes atm ocn} as well as \eqref{eq:rel Hinfty BIP RS} that the complex interpolation spaces are given by
	\begin{equation*}
	    \begin{aligned}
		    \tilde{\sX_\beta} &:= 
		    [\sX_0,D(-\tsA)]_{\beta}\\
		    &= (\rHsqbcper{2\beta}(\Omegaatm) \cap \rLqatm) \times (\rHsqbcper{2\beta}(\Omegaocn) \cap \rLqocn) \times \rHsqper{2\beta}(G)^2 \times \rHsqper{2\beta}(G)\times \rHsqper{2\beta}(G).	    
	    \end{aligned}
	\end{equation*}
    Now, the isomorphism $\sS$ from \autoref{lem:aet} is also an isomorphism in the category of Banach couples, see \cite[Section~I.2.1]{Ama:93}.
    Thanks to the functoriality of the interpolation, the complex interpolation spaces for $\sA$ are given by
		\begin{equation*}
			\begin{aligned} 
			&[\sX_0,D(-\sA)]_{\beta} 
			= \{ (\vatm, \vocn + \Llambda{0}\vice, \vice, h, a) : (\vatm, \vocn, \vice, h, a)  \in \tilde{\sX}_{\beta} \} \\
			&\hra
			 (\rHsqbcper{2\beta}(\Omegaatm) \cap \rLqatm) \times ((\rHsqbcper{2\beta}(\Omegaocn) \cap \rLqocn) + \Llambda{0} (\rHsqper{2\beta}(G)^2)) \times \rHsqper{2\beta}(G)^4.
			\end{aligned}
	\end{equation*}	
	Finally, \autoref{cor:L_0 regularity} yields $(\rHsqbcper{2\beta}(\Omegaocn) \cap \rLqocn) + \Llambda{0} (\rHsqper{2\beta}(G)^2) \hra \rHsqper{2\beta}(\Omegaocn)^2 \cap \rLqocn$.
	
	Property~(e) follows from analogous arguments, replacing complex by real interpolation and Bessel potential by Besov spaces. 
	
	Assertion~(f) is a consequence of the shape of the corresponding fractional power domain obtained in~(e).
    The last assertion follows from the compactness of the resolvent of $\tsA$ by the Rellich-Kondrachov theorem as well as the observation that the resolvent set is preserved under the similarity transform and the representation of the resolvent of $\sA$ in terms of $\tsA$.
\end{proof}

\subsection{Lipschitz estimates}
\label{ssec:nonlinearity system}

\

We establish Lipschitz estimates of the operator matrix $\sA$ given by \eqref{eq:operator matrix} and the nonlinearity $\sF$ given by \eqref{eq:nonlinearity}.
To account for the fact that we only have maximal regularity up to translation, for $\omega \in \R$ as in \autoref{thm:Hinfty operatormatrix} or \autoref{cor:functionalanalytic properties opmat}, we set $\sA_\omega := \sA - \omega$ as well as $\sF_\omega := \sF - \omega$.
Thanks to \autoref{cor:functionalanalytic properties opmat}(e), we may prove the Lipschitz estimates component-wise.

\begin{prop}\label{prop:nonlinearity system}
	Let $p,q \in (1,\infty)$ such that $\nicefrac{1}{p}+\nicefrac{1}{q} < \nicefrac{1}{2}$ and $u_0 = (v_{\atm,0},v_{\ocn,0},v_0) \in V$.
	Then there exists $r_0 > 0$ and a constant $L > 0$ such that $\oB_{\sX_\gamma}(u_0,r_0) \subset V$ and
	\begin{equation*}
	    \begin{aligned}
	        \| (\sA_\omega(u_1) - \sA_\omega(u_2))u \|_{\sX_0} &\le L \| u_1 - u_2 \|_{\sX_\gamma} \| u \|_{\sX_1},\\
	        \| \sF_\omega(u_1) - \sF_\omega(u_2) \|_{\sX_0} &\le L \| u_1 - u_2 \|_{\sX_\gamma}
	    \end{aligned}
	\end{equation*}
	for all $u_1, u_2 \in \oB_{\sX_\gamma}(u_0,r_0)$ and all $u \in \sX_1$.
\end{prop}

\begin{proof}
    We observe that $u_0 \in V$ especially satisfies that $u_0 \in W$, and that $\sX_0$ and $\sY_0$ coincide.
    The assertion of this proposition then follows by concatenating \autoref{lem:estimates rhs sea ice} and \autoref{cor:estimates bilinearity PE atm ocn}, and by exploiting \autoref{cor:functionalanalytic properties opmat}(e) with the concrete choice $\theta = 1 - \nicefrac{1}{p}$.
    	We remark that the assumptions on $p$ and $q$ imply in particular that $\theta > \nicefrac{1}{2} + \nicefrac{1}{2q}$.
\end{proof}

\section{Local Well-Posedness}
\label{sec:proofs local}

This section is dedicated to showing local strong well-posedness of the coupled system.

\begin{proof}[Proof of \autoref{thm:main thm}]
    It follows from \autoref{cor:functionalanalytic properties opmat}(c) that $-\sA_\omega$ has the property of maximal $\rLp$-regularity on $\sX_0$, and \autoref{prop:nonlinearity system} yields that suitable Lipschitz estimates are valid for $\sA_\omega$ and $\sF_\omega$.
    The assertion then follows from the local existence theorem for quasilinear evolution equations, see \cite[Theorem~5.1.1]{PS:16}, and the fact that finding a solution to the quasilinear abstract Cauchy problem \eqref{eq:quasilinear Cauchy problem} is equivalent to solving the coupled system of PDEs \eqref{eq:coupled system} by the reformulation in \autoref{sec:coupled system as qee}.
\end{proof}

\section{Global Well-Posedness close to Constant Equilibria}
\label{sec:proofs global}

In this section, we verify that a slightly simplified version of the coupled system \eqref{eq:coupled system}, or, equivalently,~\eqref{eq:quasilinear Cauchy problem}, is globally well-posed for
initial data close to constant equilibria.
To this end, we employ the so-called generalized principle of linearized stability, see \cite[Section~5.3]{PS:16}.

Let $p,q \in (1,\infty)$ such that $\nicefrac{1}{p} + \nicefrac{1}{q} < \nicefrac{1}{2}$. We investigate equilibria in the situation that the external forces $\fatm$, $\focn$ and $g \nablaH H$ vanish, the melting and freezing effects $\Sh$ and $\Sa$ are neglected, and the effect of the atmospheric wind on the sea ice in the form of $\tatm$ is not present, i.\ e., $ \fatm = \focn = g \nablaH H = \tatm = 0$ and $\Sh = \Sa = 0$.

Similarly as in \autoref{ssec:nonlinearity system}, see also \cite[Sections~6 and 7]{BDHH:22}, we verify that
\begin{equation*}
    (\sA,\sF_s) \in \rC^1(V,\sL(\sX_1,\sX_0) \times \sX_0)
\end{equation*}
for an open subset $V \subset \sX_\gamma$, where $\sF_s$ is given by
\begin{equation*}
    \sF_s(\vatm,\vocn,\vice,h,a) = \begin{pmatrix}
			\bilinatm  \\
			\bilinocn  \\
			\bilinHice\\
			\divH(\vice h)\\
			\divH(\vice a)
		\end{pmatrix}.
\end{equation*}
By $\sE \subset V \cap \sX_1$, we denote the set of equilibrium solutions to
\begin{equation}\label{eq:syst for equilibrium}
    u' - \sA(u)u + \sF_s(u) = 0, \quad t > 0, \quad u(0) = u_0,
\end{equation}
where $u = (\vatm,\vocn,\vice,h,a)$ again represents the principal variable of the system.
The property of being an equilibrium solution $u \in \sE$ is equivalent to $u \in V \cap \sX_1$ as well as $\sA(u)u = \sF_s(u)$.

We find that $u_* = (0,0,0,h_*,a_*)$, with $h_* \in (\kappa_1,\kappa_2)$ and $a_* \in (0,1)$ constant in space and time, is an equilibrium solution of \eqref{eq:syst for equilibrium}, because it satisfies the coupling condition $\trGao \vocn = \vice$ on $\Gao$, yielding that $u_* \in V \cap \sX_1$, and it holds that $\sA(u_*)u_* = 0 = \sF_s(u_*)$.

Next, we compute the total linearization of \eqref{eq:syst for equilibrium} at $u_*$ taking the shape
\begin{equation*}
    \sA_0 u = \sA(u_*)u + (\sA'(u_*)u)u_* - \sF_s'(u_*)u, \quad \for u \in \sX_1.
\end{equation*}
Abbreviating $P(h_*,a_*)$ by $P_*$, i.\ e., $P_* = p^* h_* \exp(-c(1-a_*))$, we find that $\sA(u_*)u$ is of the form
\begin{equation}\label{eq:sA(u_*)}
    \sA(u_*) u
    = \begin{pmatrix}
        \Aatm \vatm\\
        \Amocn \vocn\\
        -\Coi(h_*)  \ddnGao \vocn + \AH(u_*)\vice + \Bh(u_*) h + \Ba(u_*)a\\
        \Deltah h\\
        \Deltaa a
    \end{pmatrix},
\end{equation}
where
\begin{equation*}
    \begin{aligned}
        \Coi(h_*)  \ddnGao \vocn &= \frac{\muocn}{\rice h_*}  \ddnGao \vocn, \quad (\AH(u_*)\vice)_i = \frac{P_*}{2 \rice h_* \delta^{\nicefrac{1}{2}}} \sum_{j,k,l=1}^2 \S_{ij}^{kl} \partial_k \partial_l v_{\ice,j},\\
        \Bh(u_*)h &= -\frac{\partial_h P_*}{2 \rice h_*} \nablaH h, \qandq \Ba(u_*)a = -\frac{\partial_a P_*}{2 \rice h_*} \nablaH a.
    \end{aligned}
\end{equation*}
Moreover, the shape of $u_*$ yields that $(\sA'(u_*)u)u_* = 0$ for all $u \in \sX_1$, and we compute that $\sF_s'(u_*)u$ is given by $\sF_s'(u_*)u = (0,0,0,h_* \divH \vice,a_* \divH \vice)^\top$.
This results in the total linearization
\begin{equation}\label{eq:tot lin}
    \sA_0 u = \sA(u_*) u - \sF_s'(u_*)u = \begin{pmatrix}
        \Aatm \vatm\\
        \Amocn \vocn\\
        -\Coi(h_*)  \ddnGao \vocn + \AH(u_*)\vice + \Bh(u_*) h + \Ba(u_*)a\\
        \Deltah h - h_* \divH \vice\\
        \Deltaa a - a_* \divH \vice
    \end{pmatrix}
\end{equation}
with domain $D(\sA_0) = \sX_1$.
At this stage, we recall from \eqref{eq:rel of sX_1} that for $v \in \sX_1$, it especially follows that $\vocn \in \rW^{2,q}(\Omegaocn)^2$.
In the following, we make use of this improved regularity.

The next step is to verify that an equilibrium $u_*$ of the above shape is normally stable in the sense of \cite[Theorem~5.3.1]{PS:16}.
To this end, we first investigate the spectral properties of $\sA_0$.

\begin{lem}\label{lem:spectral props tot lin}
If $u_* = (0,0,0,h_*,a_*)$ with $h_* \in (\kappa_1,\kappa_2)$ and $a_* \in (0,1)$ constant in space and time, then the total linearization $\sA_0$ with domain $D(\sA_0) = \sX_1$ satisfies $\sigma(\sA_0) \setminus \{0\} \subset \C_-$.   
\end{lem}

\begin{proof}
As a preparation, we first calculate some integrals for $u = (\vatm,\vocn,\vice,h,a) \in \sX_1$.
Using an integration by parts, invoking the periodic boundary conditions on the lateral boundary, the homogeneous Neumann boundary conditions on the upper and lower boundary as well as the condition that the horizontal divergence of the vertically averaged $\vatm$ vanishes for $\vatm \in D(\Aatm)$, we obtain that
\begin{equation}\label{eq:int by parts Aatm}
\begin{aligned}
    -\scalarprodLtwoatm{\Aatm \vatm}{\vatm} &= \scalarprodLtwoatm{\Delta \vatm + \nablaH \pi}{\vatm}\\
    &= \intatm{|\nabla \vatm|^2} = \| \nabla \vatm \|_{\rLtwoOatm}^2.
\end{aligned}
\end{equation}

Similarly, using that $\trGao\vocn = \vice$ on $\Gao = G$ as well as $\trGab \vocn = 0$ on $\Gab$ and employing Poincar\'e's inequality, we infer that
\begin{equation}\label{eq:int by parts Amocn}
    \begin{aligned}
        -\scalarprodLtwoocn{\Amocn \vocn}{\vocn}
        &= \intocn{|\nabla \vocn |^2} - \int_{\Gao} \ddnGao \vocn \cdot \trGao\vocn \d \xH\\
        &\ge C_1 \| \vocn \|_{\rH^1(\Omegaocn)}^2 - \scalarprodLtwoG{\ddnGao \vocn}{\vice}.
    \end{aligned}
\end{equation}

Next, from \cite[(4.4) and (7.4)]{BDHH:22} we recall that
\begin{equation*}
    \sum_{i,j,k,l=1}^2 \mathbb{S}_{ij}^{kl} \dk{l} v_{\ice,j} \dk{k} v_{\ice,i} = \tri^2(\nablaH \vice) \ge \frac{2}{e^2} |\eps(\vice)|^2.
\end{equation*}
Together with an integration by parts, the periodic boundary conditions and Korn's inequality, the latter estimate yields that
\begin{equation}\label{eq:int by parts AH}
\begin{aligned}
    -\scalarprodLtwoG{\AH(u_*) \vice}{\vice}
    &= -\frac{P_*}{2 \rice h_* \delta^{\nicefrac{1}{2}}} \intG{\sum_{i,j,k,l=1}^2 \mathbb{S}_{ij}^{kl} \dk{k} \dk{l} v_{\ice,j} v_{\ice,i}}\\
    &\ge C_2 \frac{P_*}{2 \rice h_* \delta^{\nicefrac{1}{2}}} \| \nablaH \vice \|_{\rLtwoG}^2.
\end{aligned}
\end{equation}
Using the periodic boundary conditions in another integration by parts, we conclude that
\begin{equation}\label{eq:eq:int by parts hor laplacians}
    -\scalarprodLtwoG{\DeltaH h}{h}
    = \| \nablaH h \|_{\rLtwoG}^2 \quad \text{and} \quad 
    -\scalarprodLtwoG{\DeltaH a}{a}
    = \| \nablaH a \|_{\rLtwoG}^2.
\end{equation}

For $u = (\vatm,\vocn,\vice,h,a)$, to determine the spectrum of $\sA_0$, we test the equation $(\lambda - \sA_0)u = 0$ with
\begin{equation}\label{eq:choice test fct}
    (\vatm,\vocn,c_3 \vice, c_4 h, c_5 a), \quad \text{where} \quad c_3 = \frac{1}{\Coi(h_*)}, \enspace c_4 = c_3 \frac{\partial_h P_*}{2 \rice h_*^2} \text{ and } c_5 = c_3 \frac{\partial_a P_*}{2 \rice h_* a_*}.
\end{equation}
Taking into account that $h_* \in (\kappa_1,\kappa_2)$ and $a_* \in (0,1)$, we observe that $c_3, c_4, c_5 > 0$.
Integrating by the vector parts, exploiting the periodic boundary conditions on the lateral boundary and inserting \eqref{eq:int by parts Aatm}, \eqref{eq:int by parts Amocn}, \eqref{eq:int by parts AH} as well as \eqref{eq:eq:int by parts hor laplacians}, we then obtain by virtue of the choice of $c_3, c_4$ and $c_5$ in \eqref{eq:choice test fct}
\begin{equation*}
    \begin{aligned}
            0 &= \lambda \| \vatm \|_{\rLtwoatm}^2 + \lambda \| \vocn \|_{\rLtwoocn}^2 + \lambda c_3 \| \vice \|_{\rLtwoG}^2 + \lambda c_4 \| h \|_{\rLtwoG}^2 + \lambda c_5 \| a \|_{\rLtwoG}^2\\
    &\quad -\scalarprodLtwoatm{\Aatm \vatm}{\vatm} - \scalarprodLtwoocn{\Amocn \vocn}{\vocn} + \Coi(h_*) c_3 \scalarprodLtwoG{\ddnGao \vocn}{\vice}\\
    &\quad - c_3 \scalarprodLtwoG{\AH(u_*) \vice}{\vice} - c_3 \scalarprodLtwoG{\Bh(u_*) h}{\vice} - c_3 \scalarprodLtwoG{\Ba(u_*) a}{\vice}\\
    &\quad - c_4 \drh \scalarprodLtwoG{\DeltaH h}{h} + c_4 h_* \scalarprodLtwoG{\divH \vice}{h} - c_5 \dra \scalarprodLtwoG{\DeltaH a}{a} + c_5 a_* \scalarprodLtwoG{\divH \vice}{a}\\
    &\ge  \lambda \| \vatm \|_{\rLtwoatm}^2 + \lambda \| \vocn \|_{\rLtwoocn}^2 + \lambda c_3 \| \vice \|_{\rLtwoG}^2 + \lambda c_4 \| h \|_{\rLtwoG}^2 + \lambda c_5 \| a \|_{\rLtwoG}^2\\
    &\quad + C\left(\| \nabla \vatm \|_{\rLtwoOatm}^2 + \| \vocn \|_{\rH^1(\Omegaocn)}^2 + \| \nablaH \vice \|_{\rLtwoG}^2 + \| \nablaH h \|_{\rLtwoG}^2 + \| \nablaH a \|_{\rLtwoG}^2\right).
    \end{aligned}
\end{equation*}
The above relation can only hold provided $\lambda$ is real and $\lambda \le 0$.
Invoking \autoref{cor:functionalanalytic properties opmat}(g) on the $q$-independence of the spectrum of $\sA$, which carries over to $\sA_0$, we infer that indeed, $\sigma(\sA_0) \setminus \{0\} \subset \C_-$.
\end{proof}

\begin{lem}\label{lem:set of equilibria and tangent space}
Near $u_* = (0,0,0,h_*,a_*)$ with $h_* \in (\kappa_1,\kappa_2)$ and $a_* \in (0,1)$ constant in space and time, the set of equilibria $\sE$ is a $\rC^1$-manifold in $\sX_1$, and the tangent of $\sE$ at $u_*$ is isomorphic to $N(\sA_0)$.   
\end{lem}

\begin{proof}
We take into account equilibria $u = (\vatm,\vocn,\vice,h,a) \in V \cap \sX_1$ with $\| u - u_* \|_{\sX_\gamma} < r$ for given $r > 0$.
Multiplying the sea ice momentum part in the resulting stationary equation $0 = -\sA(u)u + \sF_s(u)$ by $\rice h$, we find that $u$ satisfies the equation
\begin{equation}\label{eq:mod eq for equilibrium}
    0 = \begin{pmatrix}
        -\Aatm \vatm + \bilin{\vatm}{\vatm}\\
        -\Amocn \vocn + \bilin{\vocn}{\vocn}\\
        \rice h\left(\Coi(h)  \ddnGao \vocn - \AH(\vice,h,a) \vice - \Bh(h,a)h - \Ba(h,a)a + \bilinHice)\right)\\
        -\Deltah h + \divH(\vice h)\\
        -\Deltaa a + \divH(\vice a)
    \end{pmatrix}.
\end{equation}

Again, we require some preparation.
First, we consider $u_i = (v_i,w_i)$ as well as $\Omega_i$, where $i=\atm$ or $i=\ocn$ depending on whether we study the terms associated to the atmosphere or the ocean.
For simplicity, we drop the index for the following computation, and we observe that $\div u = 0$ is valid in both cases.
Using this condition, the divergence theorem, the periodic boundary conditions on the lateral boundary as well as $w = 0$ on the upper and lower boundary, see \eqref{eq:boundary conditions w}, we obtain
\begin{equation*}
    \intO{(u \cdot \nabla)v \cdot v} = \frac{1}{2} \intO{\div(|v|^2 u)} = \frac{1}{2}\intpO{|v|^2 \tbinom{v}{w} \cdot \Vec{n}} = 0,
\end{equation*}
where $\Vec{n}$ denotes the outer normal vector.
For the oceanic term, it is crucial to note that $\Vec{n} = (0,0,\pm1)^\top$ on the upper and lower boundary, so the coupling condition with sea ice on $\Gao$ does not come into play.
For a similar result, we also refer to \cite[Lemma~6.3]{HK:16}.
By $\bilinu{u}{v} = \bilin{v}{v}$, we then get
\begin{equation}\label{eq:canc law pe}
\begin{aligned}
        \intatm{(\bilin{\vatm}{\vatm}) \cdot \vatm} &= 0, \quad \text{and}\\ 
        \intocn{(\bilin{\vocn}{\vocn}) \cdot \vocn} &= 0.
\end{aligned}
\end{equation}

We mainly write $\eps = \eps(\vice)$ in the sequel for convenience.
For $u \in V$, we especially deduce that
$    P(h,a) \ge p^* \kappa \exp{(-c)} =: P_{**}$ as well as $ \frac{1}{\trid(\eps)} \ge \frac{1}{\sqrt{\delta + c_e r^2}}$
for a suitable constant $c_e > 0$, see also the proof of Lemma~7.2 in \cite{BDHH:22}.
We emphasize that $P_{**}$ is independent of $u$, $\delta$ and $r$.
Using an integration by parts in conjunction with the periodic boundary conditions, exploiting the previous estimates, invoking $\eps^\top \mathbb{S} \eps = \tri^2(\eps) \ge 0$ by virtue of \cite[(4.4)]{BDHH:22} and employing Korn's inequality, we infer that
\begin{equation}\label{eq:int modified AH}
    \begin{aligned}
        -\scalarprodLtwoG{\rice h \AH(\vice,h,a) \vice}{\vice}
        &= -\intG{\divH\left(\frac{P(h,a)}{2} \frac{\mathbb{S} \eps}{\trid(\eps)}\right) \cdot \vice}\\
        &\ge \frac{P_{**}(1-\nicefrac{1}{e^2})}{2 \sqrt{\delta + c_e r^2}}\left( \| \divH \vice \|_{\rLtwoG}^2 + \| \eps(\vice)\|_{\rLtwoG}^2\right)\\
        &\ge \frac{C_K}{\sqrt{\delta + c_e r^2}} \| \nablaH \vice \|_{\rLtwoG}^2
    \end{aligned}
\end{equation}
for a suitable constant $C_K > 0$ emerging from Korn's inequality and incorporating $P_{**}(1-\nicefrac{1}{e^2})$.

For $u \in V \cap \sX_1$ with $\| u - u_* \|_{\sX_\gamma} < r$, we test the above equation \eqref{eq:mod eq for equilibrium} with
 $(\vatm,\vocn,c_3 \vice, c_4 h, c_5 a)$, where $c_3 = \frac{1}{\muocn }$, $c_4 = c_3 \frac{p^* \exp(-c(1-a_*))}{2 h_*}$, and $c_5 = c_3 \frac{c p^* h_* \exp(-c(1-a_*))}{2 a_*}$.
Making use of \eqref{eq:int by parts Aatm}, \eqref{eq:int by parts Amocn}, \eqref{eq:eq:int by parts hor laplacians}, \eqref{eq:canc law pe} as well as \eqref{eq:int modified AH}, remarking that
\begin{equation*}
    -\scalarprodLtwoG{\ddnGao \vocn}{\vice} + c_3 \muocn \scalarprodLtwoG{\ddnGao \vocn}{\vice} = 0
\end{equation*}
in view of the choice of $c_3$ above and integrating by parts, we obtain
\begin{equation}\label{eq:testing the equilibrium eq}
    \begin{aligned}
        0 &\ge \| \nabla \vatm \|_{\rLtwoOatm}^2 + C_1 \| \vocn \|_{\rH^1(\Omegaocn)}^2 + \frac{c_3 C_K}{\sqrt{\delta + c_e r^2}} \| \nablaH \vice \|_{\rLtwoG}^2 + c_4 \drh \| \nablaH h \|_{\rLtwoG}^2\\
        &\quad + c_5 \dra \| \nablaH a \|_{\rLtwoG}^2 + c_3 \rice \intG{h(\bilinHice) \cdot \vice}\\
        &\quad + \intG{\left(c_3 \frac{\partial_h P(h,a)}{2} - c_4 h\right) \nablaH h \cdot \vice} + \intG{\left(c_3 \frac{\partial_a P(h,a)}{2} - c_5 a\right) \nablaH a \cdot \vice}.
    \end{aligned}
\end{equation}

It remains to verify that the terms without sign in \eqref{eq:testing the equilibrium eq} can be absorbed.
Arguing as in the proof of Lemma~7.2 in \cite{BDHH:22}, and using \autoref{cor:functionalanalytic properties opmat}(e) for an embedding of the trace space $\sX_\gamma$, we find that there is a constant $C_* > 0$ such that
\begin{equation*}
\begin{aligned}
    \Big\| c_3 \frac{\partial_h P(h,a)}{2} - c_4 h \Big\|_{\rLinftyG} &= \frac{c_3}{2} \Big\| p^* \exp(-c(1-a)) - p^* \frac{h}{h_*} \exp(-c(1-a_*)) \Big\|_{\rLinftyG} &&\le C_* r, \andtext\\
    \Big\| c_3 \frac{\partial_a P(h,a)}{2} - c_5 a \Big\|_{\rLinftyG} &= \frac{c_3}{2} \Big\| c p^* h \exp(-c(1-a)) - c p^* h_* \frac{a}{a_*} \exp(-c(1-a_*)) \Big\|_{\rLinftyG} \hspace*{-1em} &&\le C_* r.
\end{aligned}
\end{equation*}
On the other hand, it follows from $\trGao \vocn = \vice$ on $\Gao$ as well as the continuity of the trace from $\rH^1(\Omegaocn)$ to $\rL^2(\Gao)$ that $\| \vice \|_{\rLtwoG}^2 = \| \trGao \vocn \|_{\rL^2(\Gao)}^2 \le C_\tr \| \vocn \|_{\rH^1(\Omegaocn)}^2$.
A concatenation of the previous two observations, H\"older's inequality and Young's inequality then leads to
\begin{equation}\label{eq:absorbing lower orders}
\begin{aligned}
    \intG{\Big(c_3 \frac{\partial_h P(h,a)}{2} - c_4 h\Big) \nablaH h \cdot \vice} &\ge - C_{**} r \left(\| \nablaH h \|_{\rLtwoG}^2 + \| \vocn \|_{\rH^1(\Omegaocn)}^2\right), \andtext\\
    \intG{\Big(c_3 \frac{\partial_a P(h,a)}{2} - c_5 a\Big) \nablaH a \cdot \vice} &\ge - C_{**} r \left(\| \nablaH a \|_{\rLtwoG}^2 + \| \vocn \|_{\rH^1(\Omegaocn)}^2\right)
\end{aligned}
\end{equation}
for another constant $C_{**} > 0$.
Similarly, we conclude
\begin{equation*}
    \begin{aligned}
        0 &\ge \| \nabla \vatm \|_{\rLtwoOatm}^2 + \left(C_1 - 2 C_{**} r - C_{***} (r + h_*) r\right) \| \vocn \|_{\rH^1(\Omegaocn)}^2\\
        &\quad + \left(\frac{c_3 C_K}{\sqrt{\delta + c_e r^2}} - C_{***}(r + h_*) r\right) \| \nablaH \vice \|_{\rLtwoG}^2 + (c_4 \drh - C_{**}r) \| \nablaH h \|_{\rLtwoG}^2\\
        &\quad + (c_5 \dra - C_{**}r) \| \nablaH a \|_{\rLtwoG}^2.
    \end{aligned}
\end{equation*}
Choosing $r$ sufficiently small, we deduce therefrom that
\begin{equation*}
    \begin{aligned}
        0 &\ge C\left(\| \nabla \vatm \|_{\rLtwoOatm}^2 + \| \vocn \|_{\rH^1(\Omegaocn)}^2 + \| \nablaH \vice \|_{\rLtwoG}^2 + \| \nablaH h \|_{\rLtwoG}^2 + \| \nablaH a \|_{\rLtwoG}^2\right)
    \end{aligned}
\end{equation*}
for some constant $C > 0$.
Consequently, for $u = (\vatm,\vocn,\vice,h,a) \in \sE$ with $u \in V \cap \sX_1$ and $\| u - u_* \|_{\sX_\gamma} < r$ for $r > 0$ sufficiently small, it holds that $\vatm$, $\vice$, $h$ and $a$ are constant and $\vocn = 0$.
As $\trGao \vocn = \vice$ on $\Gao = G$, it follows that $\vice = 0$ as well.

On the other hand, setting $\lambda = 0$ in the proof of \autoref{lem:spectral props tot lin}, we find that an element $u_0 \in N(\sA_0)$, with $u_0 = (v_{\atm,0}, v_{\ocn,0}, v_{\ice,0},h_0,a_0)$, has to satisfy that $v_{\atm,0}$, $v_{\ice,0}$, $h_0$ as well as $a_0$ are constant and $v_{\ocn,0} = 0$, implying that also $v_{\ice,0} = 0$ is valid due to the coupling condition.
As an element of this shape is especially contained in $N(\sA_0)$, we find that $\sE = N(\sA_0)$ holds in a neighborhood of $u_*$.

In summary, it is in particular valid that near $u_*$, the set of equilibria $\sE$ is a $\rC^1$-manifold in $\sX_1$ of dimension~$4$ and that the tangent space for $\sE$ at $u_*$ is isomorphic to $N(\sA_0)$.
\end{proof}

\begin{lem}\label{lem:0 semi-simple ev}
For the total linearization $\sA_0$ with domain $D(\sA_0) = \sX_1$, it holds that $0$ is a semi-simple eigenvalue of $\sA_0$, meaning that $N(\sA_0) \oplus R(\sA_0) = X_0$.    
\end{lem}

\begin{proof}
First, the structure of $\sA_0$ from \eqref{eq:tot lin} yields that $\Aatm$ can be investigated separately.
Using the representation of the hydrostatic Stokes operator with Neumann boundary conditions on the upper and lower boundary $\Aatm$ from the proof of Theorem~3.1 in \cite{GGHHK:17}, we find that $\Aatm$ inherits the property that $0$ is a semi-simple eigenvalue from the respective Neumann-Laplacian operator.
It is thus sufficient to consider the remaining matrix $\sA_0^{\mathrm{rem}}$ given by
\begin{equation*}
    \sA_0^{\mathrm{rem}} \begin{pmatrix} \vocn\\ \vice\\ h\\ a\end{pmatrix} = \begin{pmatrix}
        \Amocn \vocn\\
        -\Coi(h_*)  \ddnGao \vocn + \AH(u_*)\vice + \Bh(u_*) h + \Ba(u_*)a\\
        \Deltah h - h_* \divH \vice\\
        \Deltaa a - a_* \divH \vice
    \end{pmatrix}
\end{equation*}
on $\sX_0^\mathrm{rem} = \rLqocn \times \rLq(G)^2 \times \rLq(G) \times \rLq(G)$ with adjusted domain $D(\sA_0^{\mathrm{rem}})$.

In the sequel, we denote by $\rLq_0(G)$ the space of all $\rLq$-functions on $G$ with mean value zero, i.\ e.,
\begin{equation*}
    \frac{1}{|G|} \intG{b} = 0 \quad \text{for} \quad b \in \rLq_0(G).
\end{equation*}
We then study $\sA_{0,r}^{\mathrm{rem}}$, resulting from restricting $\sA_0^{\mathrm{rem}}$ to $\sX_{0,r}^\mathrm{rem} = \rLqocn \times \rLq(G)^2 \times \rLq_0(G) \times \rLq_0(G)$.
As in the proof of \autoref{lem:spectral props tot lin} for $\lambda = 0$, we obtain
\begin{equation*}
    0 \ge C\left(\| \vocn \|_{\rH^1(\Omegaocn)}^2 + \| \nablaH \vice \|_{\rLtwoG}^2 + \| \nablaH h \|_{\rLtwoG}^2 + \| \nablaH a \|_{\rLtwoG}^2\right)
\end{equation*}
when testing the equation $\sA_0^{\mathrm{rem}} (\vocn,\vice,h,a) = 0$ for $(\vocn,\vice,h,a) \in D(\sA_{0,r}^{\mathrm{rem}}) = D(\sA_0^\mathrm{rem}) \cap \sX_{0,r}^\mathrm{rem}$ suitably.
We deduce therefrom that $\vocn = 0$ and $\vice$, $h$ as well as $a$ are constant so that $\trGao \vocn = \vice$ on $\Gao$ and $h$, $a \in \rLq_0(G)$ yield that $\vice = 0$ as well as $h = a = 0$.
Consequently, $0$ is not an eigenvalue of $\sA_{0,r}^{\mathrm{rem}}$.
The compact resolvent, see \autoref{cor:functionalanalytic properties opmat}(g), which carries over to the present setting, then implies that $0 \in \rho(\sA_{0,r}^{\mathrm{rem}})$.

By the above argument, this time applied to $\sX_0^\mathrm{rem} = \rLqocn \times \rLq(G)^2 \times \rLq(G) \times \rLq(G)$ as the underlying ground space, we get that $N(\sA_0^{\mathrm{rem}}) = \{0\} \times \{0\} \times \R \times \R$.
Thus, to show that $\sX_0^\mathrm{rem} = N(\sA_0^{\mathrm{rem}}) + R(\sA_0^{\mathrm{rem}})$, 
it is sufficient to verify that $\rLqocn \times \rLq(G)^2 \times \rLq_0(G) \times \rLq_0(G) \subset R(\sA_0^{\mathrm{rem}})$.
To this end, we consider $f = (\focn,\fice,\frh,\fra) \in \rLqocn \times \rLq(G)^2 \times \rLq_0(G) \times \rLq_0(G)$.
By virtue of $0 \in \rho(\sA_{0,r}^{\mathrm{rem}})$, there exists $(\vocn,\vice,h,a) \in D(\sA_{0,r}^{\mathrm{rem}})$ such that
\begin{equation*}
    \sA_0^{\mathrm{rem}} (\vocn,\vice,h,a)^\top = \sA_{0,r}^{\mathrm{rem}} (\vocn,\vice,h,a)^\top = f,
\end{equation*}
resulting in $R(\sA_0^{\mathrm{rem}}) \subset \sX_{0,r}^\mathrm{rem}$ and thus also in $\sX_0^\mathrm{rem} = N(\sA_0^{\mathrm{rem}}) + R(\sA_0^{\mathrm{rem}})$ by the above argument.

It remains to check that $N(\sA_0^{\mathrm{rem}}) \cap R(\sA_0^{\mathrm{rem}}) = \{0\}$.
For $u = (v_{\ocn,0},v_{\ice,0},h_0,a_0) \in N(\sA_0^{\mathrm{rem}}) \cap R(\sA_0^{\mathrm{rem}})$, it follows from $N(\sA_0^{\mathrm{rem}}) = \{0\} \times \{0\} \times \R \times \R$ that $u = (0,0,\crh,\cra)$ for $\crh$ and $\cra$ constant.
On the other hand, by $u \in R(\sA_0^{\mathrm{rem}})$, there exists $u' = (\vocn,\vice,h,a) \in D(\sA_0^{\mathrm{rem}})$ such that $\sA_0^{\mathrm{rem}} u' = u$.
We then consider
\begin{equation*}
    u' = \begin{pmatrix} \vocn\\ \vice\\ h\\ a \end{pmatrix} = \begin{pmatrix} \vocn\\ \vice\\ \th + \oh\\ \ta + \oa \end{pmatrix}, \quad \text{where} \quad \oh = \frac{1}{|G|} \intG{h} \quad \text{and} \quad \th = h - \oh,
\end{equation*}
and likewise for $a$.
Since $\oh$ and $\oa$ are constant and $\th$, $\ta \in \rLq_0(G)$, we infer that $(0,0,\oh,\oa) \in N(\sA_0^{\mathrm{rem}})$ as well as $(\vocn,\vice,\th,\ta) \in D(\sA_{0,r}^{\mathrm{rem}})$.
Consequently, we have
\begin{equation*}
    \begin{pmatrix} 0\\ 0\\ \crh\\ \cra \end{pmatrix} = u = \sA_0^{\mathrm{rem}} u' = \sA_0^{\mathrm{rem}} \begin{pmatrix} \vocn\\ \vice\\ \th\\ \ta \end{pmatrix} + \sA_0^{\mathrm{rem}} \begin{pmatrix} 0\\ 0\\ \oh\\ \oa \end{pmatrix} = \sA_0^{\mathrm{rem}} \begin{pmatrix} \vocn\\ \vice\\ \th\\ \ta \end{pmatrix} = \sA_{0,r}^{\mathrm{rem}} \begin{pmatrix} \vocn\\ \vice\\ \th\\ \ta \end{pmatrix},
\end{equation*}
i.\ e., $u \in R(\sA_{0,r}^{\mathrm{rem}}) = \rLqocn \times \rLq(G)^2 \times \rLq_0(G) \times \rLq_0(G)$, so it follows that $\crh = \cra = 0$ and therefore, it holds that $u = 0$.
In summary, it is valid that $N(\sA_0^{\mathrm{rem}}) \oplus R(\sA_0^{\mathrm{rem}}) = \{0\}$.
In conjunction with the above argument concerning the treatment of $\Aatm$, the assertion of the lemma follows.
\end{proof}

\begin{lem}\label{lem:max reg sA(u_*)}
For $u_*$ as above, $-\sA(u_*)$ has the property of maximal regularity on $\sX_0$.    
\end{lem}

\begin{proof}
Combining the triangular structure of $\sA(u_*)$ as in \eqref{eq:sA(u_*)} and the maximal regularity of $\Aatm$, see \cite[Corollary~3.4]{GGHHK:17}, and of the horizontal Laplacians on $G$, we find that the task reduces to establishing maximal regularity of
\begin{equation*}
    \sAoi(u_*) = \begin{pmatrix}
        \Amocn & 0\\
        -\Coi(h_*)  \ddnGao & \AH(u_*)
    \end{pmatrix},
\end{equation*}
with adjusted domain $\sAoi(u_*)$, on the resulting ground space $\rLqocn \times \rLq(G)^2$.

As a result of the above argument in conjunction with \autoref{cor:functionalanalytic properties opmat}(c), there is $\omega_0 \in \R$ such that for all $\omega > \omega_0$, it holds that $-\sAoi(u_*) + \omega$ has the property of maximal regularity. 
It can be shown that $\omega_0$ can be chosen to be equal to the spectral bound of $s(\sAoi(u_*))$ of $\sAoi(u_*)$.
With regard to the $q$-independence of the spectrum, see \autoref{cor:functionalanalytic properties opmat}, carrying over to the present situation, it suffices to study the spectrum in the $\rL^2$-case.

Testing the equation $(\lambda-\sAoi(u_*)) (\vocn,\vice)^\top = 0$ with $(\vocn,c_1 \vice)$, where $c_1 = \frac{1}{\Coi{h_*}}$, and making use of \eqref{eq:int by parts Amocn} as well as \eqref{eq:int by parts AH}, we conclude that there is a constant $C > 0$ such that
\begin{equation*}
    0 \ge \lambda \| \vocn \|_{\rLtwoocn}^2 + \lambda c_1 \| \vice \|_{\rLtwoG}^2 + C\left(\| \vocn \|_{\rH^1(\Omegaocn)}^2 + \| \nablaH \vice \|_{\rLtwoG}^2\right),
\end{equation*}
so $s(\sAoi(u_*)) < 0$ by virtue of the coupling condition $\trGao \vocn = \vice$ on $\Gao$, and the result thus follows by the above arguments.
\end{proof}

\begin{proof}[Proof of \autoref{thm:second main thm}:]
Concatenating \autoref{lem:spectral props tot lin}, \autoref{lem:set of equilibria and tangent space}, \autoref{lem:0 semi-simple ev} and \autoref{lem:max reg sA(u_*)}, we find that an equilibrium $u_*$ of the above shape is normally stable. The generalized principle of linearized stability, see \cite[Theorem~5.3.1]{PS:16}, yields the assertion of \autoref{thm:second main thm}.
\end{proof}

\medskip 

{\bf Acknowledgements.}
Tim Binz would like to thank DFG for support through project~538212014.
Felix Brandt gratefully acknowledges the support by the German National Academy of Sciences Leopoldina through the Leopoldina Fellowship Program with grant number LPDS~2024-07.
Felix Brandt and Matthias Hieber would like to thank DFG for support through project FOR~5528.

\end{document}